%% file: mf2p.tex
\documentclass[11pt,final]{article}
\usepackage[textwidth=5.5in,textheight=9.0in,centering]{geometry}
\usepackage{titlesec,titletoc}
\usepackage{mathtools}
\usepackage{url}
\usepackage{amsmath}
\usepackage{amsxtra,amscd}
\usepackage{amsfonts}
\usepackage{amssymb}
\usepackage{graphicx}
\usepackage[amsmath,hyperref,thmmarks]{ntheorem}
\usepackage{natbib}
\usepackage{verbatim}
\usepackage{mathptmx}
\usepackage{enumitem}
\usepackage[notcite,notref]{showkeys}%
\theoremnumbering{arabic}
\theoremheaderfont{\scshape}
\RequirePackage{latexsym}
\theorembodyfont{\slshape}
\theoremseparator{.}
\newtheorem{X}{X}[section]

\newtheorem{corollary}[X]{Corollary}

\newtheorem{lemma}[X]{Lemma}
\newtheorem*{*lemma}{Lemma}
\newtheorem{proposition}[X]{Proposition}
\newtheorem*{*proposition}{Proposition}
\newtheorem{theorem}[X]{Theorem}
\theorembodyfont{\upshape}
\newtheorem{aside}[X]{Aside}

\newtheorem*{*definition}{Definition}
\newtheorem{example}[X]{Example}
\newtheorem*{*example}{Example}

\newtheorem{plain}[X]{}

\theorembodyfont{\footnotesize}
\theorembodyfont{\normalsize}
\theoremstyle{nonumberplain}
\theoremheaderfont{\sc}
\theorembodyfont{\normalfont}
\theoremsymbol{\ensuremath{_\Box}}
\RequirePackage{amssymb}
\newtheorem{proof}{Proof}
\qedsymbol{\ensuremath{_\Box}}
\theoremclass{LaTeX}
\makeindex
\input{defa}

\input{defs}

\begin{document}

\title{The $p$-cohomology of algebraic varieties and special values of zeta functions }
\author{James S. Milne
\and Niranjan Ramachandran\thanks{Partly supported by NSF and Graduate Research
Board (UMD)}}
\date{October 16, 2013}
\maketitle

\begin{abstract}
The $p$-cohomology of an algebraic variety in characteristic $p$ lies
naturally in the category $D_{c}^{b}(R)$ of coherent complexes of graded
modules over the Raynaud ring (Ekedahl-Illusie-Raynaud). We study homological
algebra in this category. When the base field is finite, our results provide
relations between the the absolute cohomology groups of algebraic varieties,
log varieties, algebraic stacks, etc. and the special values of their zeta
functions. These results provide compelling evidence that $D_{c}^{b}(R)$ is
the correct target for $p$-cohomology in characteristic $p$.

\end{abstract}
\tableofcontents

\subsection{Introduction}

Each of the usual cohomology theories $X\rightsquigarrow H^{j}(X,r)$ on
algebraic varieties arises from a functor $R\Gamma$ taking values in a
triangulated category $\mathsf{D}$ equipped with a $t$-structure and a Tate
twist $N\rightsquigarrow N(r)$. The heart of $\mathsf{D}$ has a tensor
structure and, in particular, an identity object $\1$. The cohomology theory
satisfies
\begin{equation}
H^{j}(X,r)\simeq H^{j}(R\Gamma(X)(r))\text{,} \label{eq2}%
\end{equation}
and there is an absolute cohomology theory%
\begin{equation}
H_{\mathrm{abs}}^{j}(X,r)\simeq\Hom_{\mathsf{D}(k)}(\1,R\Gamma(X)(r)[j]).
\label{eq1}%
\end{equation}
(see, for example, \cite{deligne1994}, \S 3).

Let $k$ be a base field of characteristic $p$. For the $\ell$-adic \'{e}tale
cohomology, $\mathsf{D}$ is the category of bounded constructible
$\mathbb{Z}{}_{\ell}$-complexes (\cite{ekedahl1990}). For the $p$-cohomology,
it is the category $\mathsf{D}_{c}^{b}(R)$ of coherent complexes of graded
modules over the Raynaud ring. This category was defined in
\cite{illusieR1983}, and its properties were developed in \cite{ekedahl1984,
ekedahl1985, ekedahl1986}. We study homological algebra in this category and,
when $k$ is finite, we prove relations between $\Ext$s and zeta functions.

Let $k=\mathbb{F}{}_{q}$ with $q=p^{a}$. The $\Ext$ of two objects $M$, $N$ of
$\mathsf{D}_{c}^{b}(R)$ is defined by the usual formula%
\[
\Ext^{j}(M,N)=\Hom_{\mathsf{D}_{c}^{b}(R)}(M,N[j]).
\]
Using that $k$ is finite, we construct a canonical complex%
\[
E(M,N)\colon\quad\cdots\rightarrow\Ext^{j-1}(M,N)\rightarrow\Ext^{j}%
(M,N)\rightarrow\Ext^{j+1}(M,N)\rightarrow\cdots
\]
of abelian groups for each pair $M,N$ in $\mathsf{D}_{c}^{b}(R)$.

An object $P$ of $\mathsf{D}_{c}^{b}(R)$ can be regarded as a double complex
of $W_{\sigma}[F,V]$-modules. On tensoring $P$ with $\mathbb{Q}{}$ and forming
the associated simple complex, we obtain a bounded complex $sP_{\mathbb{Q}{}}$
whose cohomology groups $H^{j}(sP_{\mathbb{Q}{}})$ are $F$-isocrystals over
$k$. We define the zeta function $Z(P,t)$ of $P$ to be the alternating product
of the characteristic polynomials of $F^{a}$ acting on these $F$-isocrystals.
It lies in $\mathbb{Q}{}_{p}(t)$.

Attached to each $P$ in $\mathsf{D}_{c}^{b}(R)$, there is a bounded complex
$R_{1}\otimes_{R}^{L}P$ of graded $k$-vector spaces whose cohomology groups
have finite dimension. The Hodge numbers $h^{i,j}(P)$ of $P$ are defined to be
the dimensions of the $k$-vector spaces $H^{j}(R_{1}\otimes_{R}^{L}P)^{i}$.

Finally, we let $R\underline{\Hom}(-,-)$ denote the internal Hom in
$\mathsf{D}_{c}^{b}(R)$.

\begin{theorem}
\label{a0}Let $M,N\in\mathsf{D}_{c}^{b}(R)$ and let $P=R\underline{\Hom}%
(M,N)$. Let $r\in\mathbb{Z}{}$, and assume that $q^{r}$ is not a multiple root
of the minimum polynomial of $F^{a}$ acting on $H^{j}(sP_{\mathbb{Q}{}})$ for
any integer $j$.

\begin{enumerate}
\item The groups $\Ext^{j}(M,N(r))$ are finitely generated $\mathbb{Z}{}_{p}%
$-modules, and the alternating sum of their ranks is zero.

\item The zeta function $Z(P,t)$ of $P$ has a pole at $t=q^{-r}$ of order%
\[
\rho=\sum\nolimits_{j}(-1)^{j+1}\cdot j\cdot\rank_{\mathbb{Z}{}}%
(\Ext^{j}(M,N(r))).
\]

\item The cohomology groups of the complex $E(M,N(r))$ are finite, and the
alternating product of their orders $\chi(M,N(r))$ satisfies%
\[
\left\vert \lim_{t\rightarrow q^{-r}}Z(P,t)\cdot(1-q^{r}t)^{\rho}\right\vert
_{p}^{-1}=\chi(M,N(r))\cdot q^{\chi(P,r)}%
\]
where%
\[
\chi(P,r)=\sum_{i,j\,\,(i\leq r)}(-1)^{i+j}(r-i)\cdot h^{i,j}(P).
\]

\end{enumerate}
\end{theorem}

\noindent Here $|\cdot|_{p}$ is the $p$-adic valuation, normalized so that
$|p^{r}\frac{m}{n}|_{p}^{-1}=p^{r}$ if $m$ and $n$ are prime to $p$.

We identify the identity object of $\mathsf{D}_{c}^{b}(R)$ with the ring $W$
of Witt vectors. Then $R\underline{\Hom}(W,N)\simeq N$.

Each algebraic variety (or log variety or stack) defines several objects in
$\mathsf{D}_{c}^{b}(R)$ (see \S 6). Let $M(X)$ be one of the objects of
$\mathsf{D}_{c}^{b}(R)$ attached to an algebraic variety $X$ over $k$, and
define the absolute cohomology of $X$ to be%
\[
H_{\mathrm{abs}}^{j}(X,\mathbb{Z}{}_{p}(r))=\Hom_{\mathsf{D}_{c}^{b}%
(R)}(W,M(X)(r)[j]).
\]
The complex $E(W,M(X)(r))$ becomes,%
\[
E(X,r)\colon\quad\cdots\rightarrow H_{\mathrm{abs}}^{j-1}(X,\mathbb{Z}{}%
_{p}(r))\rightarrow H_{\mathrm{abs}}^{j}(X,\mathbb{Z}{}_{p}(r))\rightarrow
H_{\mathrm{abs}}^{j+1}(X,\mathbb{Z}{}_{p}(r))\rightarrow\cdots.
\]

\begin{theorem}
\label{b04}Assume that $q^{r}$ is not a multiple root of the minimum
polynomial of $F^{a}$ acting on $H^{j}(sM(X)_{\mathbb{Q}{}})$ for any $j$.

\begin{enumerate}
\item The groups $H_{\mathrm{abs}}^{j}(X,\mathbb{Z}_{p}(r))$ are finitely
generated $\mathbb{Z}{}_{p}$-modules, and the alternating sum of their ranks
is zero.

\item The zeta function $Z(M(X),t)$ of $M(X)$ has a pole at $t=q^{-r}$ of
order%
\[
\rho=\sum\nolimits_{j}(-1)^{j+1}\cdot j\cdot\rank_{\mathbb{Z}{}_{p}}\left(
H_{\mathrm{abs}}^{j}(X,\mathbb{Z}{}_{p}(r))\right)  \text{.}%
\]

\item The cohomology groups of the complex $E(X,r)$ are finite, and the
alternating product of their orders $\chi(X,\mathbb{Z}{}_{p}(r))$ satisfies%
\[
\left\vert \lim_{t\rightarrow q^{-r}}Z(M(X),t)\cdot(1-q^{r}t)^{\rho
}\right\vert _{p}^{-1}=\chi(X,\mathbb{Z}{}_{p}(r))\cdot q^{\chi(M(X),r)}.
\]

\end{enumerate}
\end{theorem}

Let $X$ be a smooth projective variety over $k$, and let $M(X)=R\Gamma
(X,W\Omega_{X}^{\bullet})$. Then $H^{j}(sM(X)_{\mathbb{Q}{}})=H_{\mathrm{crys}%
}^{j}(X/W)_{\mathbb{Q}{}}$ and $H_{\mathrm{abs}}^{j}(X,\mathbb{Z}_{p}(r))$ is
the group $H^{j}(X,\mathbb{Z}{}_{p}(r))$ defined in (\ref{n7}) below.
Moreover, the zeta function and the Hodge numbers of $M(X)$ agree with those
of $X$, and so, in this case, Theorem \ref{b04} becomes the $p$-part of the
main theorem of \cite{milne1986}. See p.\pageref{eq65} below.

\subsection{Remarks}

\begin{plain}
\label{a03}Let $\zeta(P,s)=Z(P,q^{-s}),$ $s\in\mathbb{C}$. Then $\rho$ is the
order of the pole of $\zeta(P,s)$ at $s=r$, and
\[
\lim_{t\rightarrow q^{-r}}Z(P,t)\cdot(1-q^{r}t)^{\rho}=\lim_{s\rightarrow
r}\zeta(P,s)\cdot(s-r)^{\rho}\cdot(\log q)^{\rho}\text{.}%
\]

\end{plain}

\begin{plain}
\label{a04}We expect that the $F$-isocrystals $H^{j}(sP_{\mathbb{Q}{}})$ are
always semisimple (so $F^{a}$ always acts semisimply) when $P$ arises from
algebraic geometry. If this fails, there will be spurious extensions over
$\mathbb{Q}{}$ that will have to be incorporated into the statement of
(\ref{a0}).
\end{plain}

\begin{plain}
\label{a04a}The statement of Theorem 0.1 depends only on $\mathsf{D}_{c}%
^{b}(R)$ as a triangulated category with a dg-lifting.
\end{plain}

\begin{plain}
\label{a05}We leave it as an (easy) exercise for the reader to prove the
analogue of (\ref{a0}) for $\ell\neq p$ (the indolent may refer to article below).
\end{plain}

\begin{plain}
\label{a06}In a second article, we apply (\ref{a0}) to study the analogous
statement in a triangulated category of motivic complexes (\cite{milneR2013b}).
\end{plain}

\subsection{Outline of the article}

In \S 1 and \S 3 we review some of the basic theory of the category
$\mathsf{D}_{c}^{b}(R)$ (Ekedahl, Illusie, Raynaud), and in \S 2 we prove a
relation between the numerical invariants of an object of $\mathsf{D}_{c}%
^{b}(R)$. In \S 4 begin the study of the homological algebra of $\mathsf{D}%
_{c}^{b}(R)$, and in \S 5 we take the ground field to be finite and prove
Theorem \ref{a0}. In the final section we study applications of Theorem
\ref{a0} to algebraic varieties.

\subsection{Notations}

Throughout, $k$ is a perfect field of characteristic $p\neq0$, and $W$ is the
ring of Witt vectors over $k$. As usual, $\sigma$ denotes the automorphism of
$W$ inducing $a\mapsto a^{p}$ modulo $p$. We use a bar to denote base change
to an algebraic closure $\bar{k}$ of $k$. For example, $\bar{W}$ denotes the
Witt vectors over $\bar{k}$. We use $\simeq$ to denote a canonical, or
specific, isomorphism.

\section{Coherent complexes of graded $R$-modules}

In this section, we review some definitions and results of Ekedahl, Illusie,
and Raynaud, for which \cite{illusie1983} is a convenient reference.

\begin{plain}
\label{a2}The \emph{Raynaud ring} is the graded $W$-algebra $R=R^{0}\oplus
R^{1}$ generated by $F$ and $V$ in degree $0$ and $d$ in degree $1$, subject
to the relations
\begin{align}
FV  &  =p=VF,\quad Fa=\sigma a\cdot F,\quad aV=V\cdot\sigma a,\label{eq5}\\
d^{2}  &  =0,\quad FdV=d,\quad ad=da\quad(a\in W). \label{eq6}%
\end{align}
In other words, $R^{0}$ is the Dieudonn\'{e} ring $W_{\sigma}[F,V]$ and $R$ is
generated as an $R^{0}$-algebra by a single element $d$ of degree $1$
satisfying (\ref{eq6}). For $m\geq1,$%
\begin{equation}
R_{m}\overset{\textup{{\tiny def}}}{=}R/(V^{m}R+dV^{m}R)\text{.} \label{eq42}%
\end{equation}

\end{plain}

\begin{plain}
\label{a3}To give a graded $R$-module $M=\bigoplus\nolimits_{i\in\mathbb{Z}{}%
}M^{i}$ is the same as giving a complex%
\[
M^{\bullet}\colon\quad\cdots\rightarrow M^{i-1}\overset{d}{\longrightarrow
}M^{i}\overset{d}{\longrightarrow}M^{i+1}\rightarrow\cdots
\]
of $W$-modules whose components $M^{i}$ are $R^{0}$-modules and whose
differentials $d$ satisfy $FdV=d$. For $n\in\mathbb{Z}{}$, $M\{n\}$ is the
graded $R$-module deduced from $M$ by a shift of degree,\footnote{Illusie et
al.\ write $M(n)$ for the degree shift of $M$, but this conflicts with our
notation for Tate twists.} i.e., $M\{n\}^{i}=M^{n+i}$ and $d_{M\{n\}}%
^{i}=(-1)^{n}d_{M}^{n+i}$. The graded $R$-modules and graded homomorphisms of
degree $0$ form an abelian category $\Mod(R)$ with derived category
$\mathsf{D}(R)$. The bifunctor $M,N\rightsquigarrow\Hom(M,N)$ of graded
$R$-modules derives to a bifunctor%
\[
R\Hom\colon\mathsf{D}(R)^{\mathrm{opp}}\times\mathsf{D}^{+}(R)\rightarrow
\mathsf{D}(\mathbb{Z}{}_{p})
\]
(denoted by $R\Hom_{R}$ in \cite{illusie1983}, 2.6.2, and \cite{ekedahl1986},
p.8, and by $R\underline{\Hom}_{R}$ in \cite{ekedahl1985}, p.73).
\end{plain}

\begin{plain}
\label{a4}A graded $R$-module is said to be \emph{elementary}
(\cite{illusie1983}, 2.2.2, p.30) if it is one of the following two types.

\begin{description}
\item[Type I] The module is concentrated in degree zero, finitely generated
over $W$, and $V$ is topologically nilpotent on it. In other words, it is a
$W_{\sigma}[F,V]$-module whose $p$-torsion submodule has finite length over
$W$, and whose torsion-free quotient is finitely generated and free over $W$
with slopes lying in the interval $[0,1[$.

\item[Type II] The module is isomorphic to%
\[
U_{l}\colon\quad\underset{\vphantom{A^{A^A}}\deg0}{\prod\nolimits_{n\geq
0}kV^{n}}\overset{d}{\longrightarrow}\underset{\vphantom{A^{A^A}}\deg
1}{\prod\nolimits_{n\geq l}kdV^{n}}%
\]
for some $l\in\mathbb{Z}$. Here $F$ (resp. $V)$ acts as zero on $U_{l}^{0}$
(resp. $U_{l}^{1}$), and $dV^{n}$ should be interpreted as $F^{-n}d$ when
$n<0$. In more detail, $U_{l}^{0}$ is the $W_{\sigma}[F,V]$-module $k[[V]]$
with $F$ acting as zero. When $l\geq0$, $U_{l}^{1}$ consists of the formal
sums%
\[
a_{l}dV^{l}+a_{l+1}dV^{l+1}+\cdots\quad\quad(a_{l}\in k),
\]
and when $l<0$, $U_{l}^{1}$ consists of the formal sums
\[
a_{-l}F^{-l}d+\cdots+a_{-1}F^{-1}d+a_{0}d+a_{1}dV+a_{2}dV^{2}+\cdots\quad
\quad(a_{l}\in k).
\]

\end{description}
\end{plain}

\begin{plain}
\label{a5}A graded $R$-module $M$ is said to be \emph{coherent} if it admits a
finite filtration $M\supset\cdots\supset0$ whose quotients are degree shifts
of elementary modules (i.e., of the form $M\{n\}$ with $M$ elementary and
$n\in\mathbb{Z}{}$). Coherent $R$-modules need not be noetherian or
artinian---the object $U_{0}$ is obviously neither.
\end{plain}

\begin{plain}
\label{a6}A complex $M$ of $R$-modules is said to be \emph{coherent} if it is
bounded with coherent cohomology. Let $\mathsf{D}_{c}^{b}(R)$ denote the full
subcategory of $\mathsf{D}(R)$ consisting of coherent complexes. Ekedahl has
given a criterion for a complex to lie in $\mathsf{D}_{c}^{b}(R)$, from which
it follows that $\mathsf{D}_{c}^{b}(R)$ is a triangulated subcategory of
$\mathsf{D}(R)$; in particular, the coherent modules form an abelian
subcategory of $\Mod(R)$ closed under extensions (\cite{illusie1983}, 2.4.8).
In more detail (ibid. 2.4), define a graded $R_{\bullet}$-module to be a
projective system%
\[
M_{\bullet}=\left(  M_{1}\leftarrow\cdots\leftarrow M_{m}\leftarrow
M_{m+1}\leftarrow\cdots\right)
\]
equipped with maps $F\colon M_{m+1}\rightarrow M_{m}$ and $V\colon
M_{m}\rightarrow M_{m+1}$ of degree zero satisfying (\ref{eq5}) and
(\ref{eq6}); here $M_{m}$ is a graded $W_{m}[d]$-module. The graded
$R_{\bullet}$-modules form an abelian category. The functor $M_{\bullet
}\rightsquigarrow\varprojlim M_{m}\colon\Mod(R_{\bullet})\rightarrow\Mod(R)$
derives to a functor%
\[
R\varprojlim\colon\mathsf{D}(R_{\bullet})\rightarrow\mathsf{D}(R).
\]
On the other hand, the functor sending a graded $R$-module $M$ to the
$R_{\bullet}$-module $(R_{m}\otimes_{R}M)_{m\geq1}$ derives to a functor%
\[
R_{\bullet}\otimes_{R}^{L}-\colon\mathsf{D}(R)\rightarrow\mathsf{D}%
(R_{\bullet})
\]
These functors compose to a functor%
\[
M\rightsquigarrow\hat{M}\colon\mathsf{D}(R)\rightarrow\mathsf{D}(R).
\]
For $M$ in $\mathsf{D}^{-}(R)$, there is a natural map $M\rightarrow\hat{M}$
inducing isomorphisms $R_{m}\otimes_{R}^{L}M\rightarrow R_{m}\otimes_{R}%
^{L}\hat{M}$ for all $m$, and $M$ is said to be \emph{complete} if this map is
an isomorphism. Ekedahl's criterion states:\bquote A bounded complex of graded
$R$-modules $M$ lies in $\mathsf{D}_{c}^{b}(R)$ if and only if $M$ is complete
and $R_{1}\otimes_{R}^{L}M$ is a bounded complex such that $H^{i}(R_{1}%
\otimes_{R}^{L}M)$ is finite-dimensional over $k$ for all $i$.\equote

\end{plain}

\begin{plain}
\label{a8}Let $T$ be the functor of graded $R$-modules such that
$(TM)^{i}=M^{i+1}$ and $T(d)=-d$, i.e., $TM=M\{1\}$ (degree shift). It is
exact and defines a self-equivalence $T\colon\mathsf{D}_{c}^{b}(R)\rightarrow
\mathsf{D}_{c}^{b}(R)$. The \emph{Tate twist} of a coherent complex of graded
$R$-modules $M$ is defined as
\[
M(r)=T^{r}(M)[-r]=M\{r\}[-r];
\]
thus $M(r)^{i,j}=M^{i+r,i-r}$ (cf. \cite{milneR2005}, \S 2).
\end{plain}

\begin{plain}
\label{a9}Following \cite{illusie1983}, 2.1, we view a complex of graded
$R$-modules
\[
M\colon\quad\quad\cdots\rightarrow M^{\bullet,j}\rightarrow M^{\bullet
,j+1}\rightarrow\cdots
\]
as a bicomplex $M^{\bullet,\bullet}$ of $R^{0}$-modules in which the first
index corresponds to the $R$-gradation. Thus the $j$th row $M^{\bullet,j}$ of
the bicomplex is a graded $R$-module and the $i$th column $M^{i\bullet}$ is a
complex of $R^{0}$-modules:%
\begin{equation}
\begin{tikzcd}[column sep=scriptsize, row sep=scriptsize] \vdots & & &\vdots & \vdots & \vdots\\ M^{\bullet j+1}\arrow{u}: && \cdots \arrow{r} & M^{i-1,j+1}\arrow{r}{d}\arrow{u} & M^{i,j+1} \arrow{r}{d}\arrow{u} & M^{i+1,j+1} \arrow{r}\arrow{u} & \cdots\\ M^{\bullet j}\arrow{u}: && \cdots \arrow{r} & M^{i-1,j}\arrow{r}{d}\arrow{u} & M^{i,j} \arrow{r}{d}\arrow{u} & M^{i+1,j} \arrow{r}\arrow{u} & \cdots\\ \vdots\arrow{u} & & & \vdots\arrow{u} &\vdots\arrow{u} & \vdots\arrow{u} & \end{tikzcd}\label{e1}%
\end{equation}
In this diagram, the squares commute, the vertical differentials commute with
$F$ and $V$, and the horizontal differentials satisfy $FdV=d$. The cohomology
modules of $M$ are obtained by passing to the cohomology in the columns:%
\[
H^{j}(M)\colon\quad\quad\cdots\rightarrow H^{j}(M^{i-1,\bullet}%
)\overset{d}{\longrightarrow}H^{j}(M^{i,\bullet})\overset{d}{\longrightarrow
}H^{j}(M^{i+1,\bullet})\rightarrow\cdots.
\]
In other words, for a complex $M=M^{\bullet,\bullet}$ of graded $R$-modules,
$H^{j}(M)$ is the graded $R$-module with $H^{j}(M)^{i}=H^{j}(M^{i,\bullet})$.

By definition, $M\{m\}[n]$ is the bicomplex with
\begin{equation}
(M\{m\}[n])^{i,j}=M^{i+m,j+n} \label{e25}%
\end{equation}
and with the appropriate sign changes on the differentials.
\end{plain}

\begin{plain}
\label{a11}With any complex $M$ of graded $R$-modules, there is an associated
simple complex $sM$ of $W$-modules with%
\[
\left(  sM\right)  ^{n}=\bigoplus\nolimits_{i+j=n}M^{i,j},\quad dx^{ij}%
=d^{\prime}x^{ij}+(-1)^{i}d^{\prime\prime}x^{ij}\text{.}%
\]
The functor $s$ extends to a functor $s\colon\mathsf{D}^{+}(R)\rightarrow
\mathsf{D}(W)$. If $M\in\mathsf{D}_{c}^{b}(M)$, then $sM$ is a perfect complex
of $W$-modules (\cite{illusie1983}, p.34).
\end{plain}

\begin{plain}
\label{b0}For a coherent complex $M$ of graded $R$-modules, the filtration of
$sM$ by the first degree defines a spectral sequence%
\begin{equation}
E_{1}^{ij}=H^{j}(M)^{i}\implies H^{i+j}(sM) \label{eq7}%
\end{equation}
called the \emph{slope spectral sequence}. The slope spectral sequence
degenerates at $E_{1}$ modulo torsion and at $E_{2}$ modulo $W$-modules of
finite length. In particular, for $r\geq2$, $E_{r}^{ij}$ is a finitely
generated $W$-module of rank equal to that of $H^{j}(M)^{i}/\mathrm{torsion}%
$\textrm{. }This was proved by \cite{bloch1977} and \cite{illusieR1983} for
the complex $M=R\Gamma(X,W\Omega_{X}^{\bullet})$ attached to a smooth complete
variety $X$, and by Ekedahl for a general $M$ (see \cite{illusie1983}, 2.5.4).
\end{plain}

\begin{plain}
\label{a12}Let $K=W\otimes\mathbb{Q}{}$ (field of fractions of $W$). Then
$K\otimes_{W}W_{\sigma}[F,V]\simeq K_{\sigma}[F]$. Recall that an
$F$-isocrystal is a $K_{\sigma}[F]$-module that is finite-dimensional as a
$K$-vector space and such that $F$ is bijective. The $F$-isocrystals form an
abelian subcategory of $\Mod(K_{\sigma}[F])$ closed under extensions, and so
the subcategory $\mathsf{D}_{\text{\textrm{iso}}}^{b}(K_{\sigma}[F])$ of
$\mathsf{D}^{b}(K_{\sigma}[F])$ consisting of bounded complexes whose
cohomology modules are $F$-isocrystals is triangulated.
\end{plain}

\begin{plain}
\label{a13}Let $M$ be a complex of graded $R$-modules with only nonnegative
first degrees, and let $F^{\prime}$ act on $M^{i,j}$ as $p^{i}F$. The
condition $FdV=d$ implies that $pFd=dF$, and so both differentials in the
diagram (\ref{e1}) commute with the action of $F^{\prime}$. Therefore $s(M)$
is a complex of $W_{\sigma}[F^{\prime}]$-modules. If $M\in\mathsf{D}_{c}%
^{b}(R)$, then $s(M)_{K}$ lies in $\mathsf{D}_{\text{\textrm{iso}}}%
^{b}(K_{\sigma}[F^{\prime}])$.\medskip\ From the degeneration of the slope
spectral sequence at $E_{1}$, we get isomorphisms%
\begin{equation}
(H^{j}(M)_{K}^{i},p^{i}F)\simeq\left(  H^{i+j}(sM)_{K}\right)  _{[i,i+1[}
\label{eq7a}%
\end{equation}
for $M\in\mathsf{D}_{c}^{b}(R)$. This can also be written\footnote{For each
$n$, we have $H^{n}(sM)_{K}=\bigoplus H^{j}(M)_{K}^{i}$ where the sum if over
pairs $(i,j)$ with $i+j=n$. Our assumption on $M$ says that $i\geq0$, and so
only $H^{n}(M)_{K}^{0}$, $H^{n-1}(M)_{K}^{i},$\ldots$,H^{0}(M)_{K}^{n}$
contribute. Each of these (with the map $F$) is an isocrystal with slopes
$[0,1)$. But with the map $p^{i}F$, the slopes of $H^{n-i}(M)_{K}^{i}$ are in
$[i,i+1[$. The slopes of distinct summands to not overlap. Hence we get
(\ref{eq7b}). Cf. \cite{illusie1983}, p.64.}%
\begin{equation}
(H^{n-i}(M)_{K}^{i},p^{i}F)\simeq\left(  H^{n}(sM)_{K}\right)  _{[i,i+1[}%
\text{.} \label{eq7b}%
\end{equation}

\end{plain}

\begin{plain}
\label{a14}A \emph{domino} $N$ is a graded $R$-module that admits a finite
filtration $N\supset\cdots\supset0$ whose quotients are elementary of type II.

Let $N$ be elementary of type II, say $N=U_{l}$. Then $N^{0}=k_{\sigma}[[V]]$,
and so $V\colon N^{0}\rightarrow N^{0}$ is injective with cokernel
$N^{0}/V=k_{\sigma}[[V]]/(V)\simeq k$. Similarly, $F\colon N^{1}\rightarrow
N^{1}$ is surjective with kernel $kdV^{l}$ ($l\geq0$) or $kF^{-l}d$ ($l<0$).

Let $N$ be a domino, and suppose that $N$ admits a filtration of length $l(N)$
with elementary quotients. Induction on $l(N)$ shows that

\begin{enumerate}
\item the map $V\colon N^{0}\rightarrow N^{0}$ is injective with cokernel of
dimension $l(N)$ (as a $k$-vector space) and $F|N^{0}$ is nilpotent;

\item the map $F\colon N^{1}\rightarrow N^{1}$ is surjective with kernel of
dimension $l(N)$ and $V|N^{1}$ is nilpotent.
\end{enumerate}

\noindent Therefore the number of quotients in such a filtration is
independent of the filtration, and equals the common dimension of the
$k$-vector spaces $N^{0}/V$ and of $\Ker\left(  F\colon N^{1}\rightarrow
N^{1}\right)  $. This number is called the \emph{dimension} of $N$.
\end{plain}

\begin{plain}
\label{a17}Let $M$ be a graded $R$-module. Then $Z^{i}%
(M)\overset{\textup{{\tiny def}}}{=}\Ker\left(  d\colon M^{i}\rightarrow
M^{i+1}\right)  $ is stable under $F$ but not in general under $V$, whereas
$B^{i}(M)\overset{\textup{{\tiny def}}}{=}\im(d\colon M^{i-1}\rightarrow
M^{i})$ is stable under $V$ but not in general under $F$. Instead, one puts%
\begin{align*}
V^{-\infty}Z^{i}  &  (M)=\{x\in M^{i}\mid V^{n}x\in Z^{i}(M)\text{ for all
}n\},\\
F^{\infty}B^{i}  &  (M)=\{x\in M^{i}\mid x\in F^{n}B^{i}(M)\text{ for some
}n\}.
\end{align*}
Then $V^{-\infty}Z^{i}$ is the largest $R^{0}$-submodule of $Z^{i}(M),$ and
$F^{\infty}B^{i}$ is the smallest $R^{0}$-submodule of $M^{i}$ containing
$B^{i}M$:%
\begin{equation}
B^{i}\subset F^{\infty}B^{i}\subset V^{-\infty}Z^{i}\subset Z^{i}.
\label{eq58}%
\end{equation}
The homomorphism of $W$-modules $d\colon M^{i}\rightarrow M^{i+1}$ factors as
\begin{equation}
\begin{tikzcd} M^{i} \arrow{r}{d}\arrow{d} &M^{i+1}\\ M^{i}/V^{-\infty}Z^{i} \arrow{r}{d} & F^{\infty}B^{i+1}\arrow{u} \end{tikzcd} \label{e24}%
\end{equation}
When $M$ is coherent, the lower row in (\ref{e24}) is an $R$-module admitting
a finite filtration whose quotients are of the form $U_{l}\{-i\}$; in other
words, $($\textrm{lower\thinspace\thinspace}\textrm{row}$\mathrm{)}\{i\}$ is a
domino (\cite{illusie1983}, 2.5.2).
\end{plain}

\begin{plain}
\label{a18}The \emph{heart } of a graded $R$-module $M$ is the graded $R_{0}%
$-module $\heartsuit(M)=\bigoplus\heartsuit^{i}(M)$ with $\heartsuit
^{i}(M)=V^{-\infty}Z^{i}/F^{\infty}B^{i}$ (see (\ref{eq58})). When $M$ is
coherent, $\heartsuit(M)$ is finitely generated as a $W$-module; moreover,
$Z^{i}/V^{-\infty}Z^{i}$ and $F^{\infty}B^{i}/B^{i}$ are of finite length, and
so%
\[
\heartsuit^{i}(M)_{K}\simeq\left(  Z^{i}(M)/B^{i}(M)\right)  _{K}%
\]
(\cite{illusie1983}, 2.5.3).
\end{plain}

\begin{example}
\label{b1}Let $X$ be a smooth variety over a perfect field $k$. The de
Rham-Witt complex%
\[
W\Omega_{X}^{\bullet}\colon\quad W\mathcal{O}{}_{X}\longrightarrow
\cdots\longrightarrow W\Omega_{X}^{i}\overset{d}{\longrightarrow}W\Omega
_{X}^{i+1}\longrightarrow\cdots
\]
is a sheaf of graded $R$-modules on $X$ for the Zariski topology. On applying
$R\Gamma$ to this complex, we get a complex $R\Gamma(X,W\Omega_{X}^{\bullet})$
of graded $R$-modules, which we regard as a bicomplex with $(i,j)$th term
$R\Gamma(X,W\Omega_{X}^{i})^{j}$. When we replace each vertical complex with
its cohomology, the $j$th row of the bicomplex becomes%
\[
R^{j}\Gamma(X,W\Omega_{X}^{\bullet})\colon\quad H^{j}(X,W\mathcal{O}{}%
_{X})\rightarrow\cdots\rightarrow H^{j}(X,W\Omega_{X}^{i}%
)\overset{d}{\rightarrow}H^{j}(X,W\Omega_{X}^{i+1})\rightarrow\cdots.
\]
The complex $R\Gamma(X,W\Omega_{X}^{\bullet})$ is bounded and complete
(\cite{illusie1983}, 2.4), and becomes $R\Gamma(X,\Omega_{X}^{\bullet})$ when
tensored with $R_{1}$, and so $R\Gamma(X,W\Omega_{X}^{\bullet})$ is coherent
when $X$ is complete. In this case, $R\Gamma(X/W)\overset{\textup{{\tiny def}%
}}{=}s(R\Gamma(X,W\Omega_{X}^{\bullet}))$ is a perfect complex of $W$-modules
such that
\[
H^{j}(R\Gamma(X/W))\simeq H_{\mathrm{crys}}^{j}(X/W)\quad(\text{isomorphism of
}W_{\sigma}[F]\text{-modules})
\]
(ibid. 1.3.5), and the slope spectral sequence (\ref{eq7}) becomes%
\begin{equation}
E_{1}^{ij}=H^{j}(X,W\Omega_{X}^{i})\implies H^{i+j}(X,W\Omega_{X}^{\bullet
})\quad\quad(\simeq H_{\mathrm{crys}}^{\ast}(X/W)\text{.} \label{eq33}%
\end{equation}

\end{example}

\section{The numerical invariants of a coherent complex}

\subsection{Definition of the invariants}

Let $M$ be a coherent graded $R$-module. The dimension of the domino attached
to $d\colon M^{i}\rightarrow M^{i+1}$ (see (\ref{a17})) is denoted by
$T^{i}(M)$. It is equal to the number of quotients of the form $U_{l}\{-i\}$
(varying $l$) in a filtration of $M$ with elementary quotients.

\begin{lemma}
\label{a1a}For a coherent graded $R$-module,
\begin{equation}
T^{i}(M)=\mathrm{length}_{W((V))}W((V))\otimes_{W[[V]]}M^{i}\text{.}
\label{eq24}%
\end{equation}

\end{lemma}

\begin{proof}
It suffices to prove this for an elementary graded $R$-module $M$. If $M$ is
elementary of type I, then $V$ is topologically nilpotent on it, and so when
we invert $V$, $M$ becomes $0$; this agrees with $T^{0}(M)=0$. If $M$ is
elementary of type II , say $M=U_{l}$, then $W((V))\otimes M^{0}\simeq W((V))$
and $W((V))\otimes M^{1}=0$, agreeing with $T^{0}(M)=1$ and $T^{i}(M)=0$ for
$i\neq0$.
\end{proof}

Let $M$ be an object of $\mathsf{D}_{c}^{b}(R)$. Ekedahl (1986,
p.14)\nocite{ekedahl1986} defines the \emph{slope numbers} of $M$ to be%
\[
m^{i,j}(M)={\mathrm{dim}}_{k}\frac{H^{j}(M)^{i}}{H^{j}(M)_{p\text{-}%
\mathrm{tors}}^{i}+V\left(  H^{j}(M)^{i}\right)  }+{\mathrm{dim}}_{k}%
\frac{H^{j+1}(M)^{i-1}}{H^{j+1}(M)_{p\text{-}\mathrm{tors}}^{i-1}+F\left(
H^{j+1}(M)^{i-1}\right)  }%
\]
where $X_{p\text{-}\mathrm{tors}}$ denotes the torsion submodule of $X$
regarded as a $W$-module. Set%
\[
T^{i,j}(M)=T^{i}(H^{j}(M)).
\]
Ekedahl (ibid., p.85) defines the \emph{Hodge-Witt numbers} of $M$ to be
\[
h_{W}^{i,j}(M)=m^{i,j}(M)+T^{i,j}(M)-2T^{i-1,j+1}(M)+T^{i-2,j+2}(M)
\]
(see also \cite{illusie1983}, 6.3). Note that the invariants $m^{i,j}(M)$ and
$T^{i,j}(M)$ (hence also $h_{W}^{i,j}(M)$) depend only on the finite sequence
$(H^{j}(M))_{j\in\mathbb{Z}{}}$ of graded $R$-modules. It follows from
(\ref{e25}) that%
\begin{equation}
h_{W}^{i,j}(M\{m\}[n])=h_{W}^{i+m,j+n}(M). \label{eq56}%
\end{equation}
In particular (see \ref{a8}),%
\begin{equation}
h_{W}^{i,j}(M(r))=h_{W}^{i+r,j-r}(M). \label{eq57}%
\end{equation}

\begin{example}
\label{b2}We compute these invariants for certain $M\in\mathsf{D}_{c}^{b}(R)$.

(a) Suppose that $H^{j}(M)^{i}$ has finite length over $W$ for all $i,j$. Then
$H^{j}(M)^{i}=H^{j}(M)_{p\text{\textrm{-}}\mathrm{tors}}^{i}$, and so
$m^{i,j}(M)$ is zero for all $i$, $j$. Moreover $V$ is nilpotent on
$H^{j}(M)^{i}$, and so $T^{i,j}(M)=0$. It follows that $h_{W}^{i,j}(M)$ is
also zero for all $i$, $j$.

(b) Suppose that%
\[
H^{j}(M)^{i}=\left\{
\begin{array}
[c]{ll}%
R^{0}/R^{0}(F^{r-s}-V^{s}) & \text{if }(i,j)=(i_{0},j_{0})\\
0 & \text{otherwise}%
\end{array}
\right.
\]
for some $r>s\geq0$. Then%
\[
m^{i.j}(M)=\left\{
\begin{array}
[c]{ll}%
\dim_{k}(W_{\sigma}[F]/(F^{r-s})=r-s & \text{if }(i,j)=(i_{0},j_{0})\\
\dim_{k}(W_{\sigma}[V]/(V^{s})=s & \text{if }(i,j)=(i_{0}+1,j_{0}-1)\\
0 & \text{otherwise.}%
\end{array}
\right.
\]
Note that
\[
\left(  R^{0}/R^{0}(F^{r-s}-V^{s})\right)  \otimes K\simeq K_{\sigma
}[F]/(F^{r}-p^{s}),
\]
which is an $F$-isocrystal of slope $\lambda=s/r$ with multiplicity $m=r$. As
the dominoes attached to the $H^{j}(M)^{i}$ are obviously all zero, we see
that%
\[
h_{W}^{i,j}(M)=m^{i.j}(M)=\left\{
\begin{array}
[c]{ll}%
m(1-\lambda)\quad & \text{if }(i,j)=(i_{0},j_{0})\\
m\lambda & \text{if }(i,j)=(i_{0}+1,j_{0}-1)\\
0 & \text{otherwise}%
\end{array}
\right.
\]
where $\lambda$ is the unique slope of the $F$-isocrystal $(H^{j_{0}}%
(M)_{K}^{i_{0}},F)$ and $m$ is its multiplicity.

(c) Suppose that%
\[
\left(  H^{j_{0}}(M)^{i_{0}}\overset{d}{\longrightarrow}H^{j_{0}}(M)^{i_{0}%
+1}\right)  =\left(  \underset{\vphantom{A^{A^A}}\deg\text{ }i_{0}%
}{\prod\nolimits_{n\geq0}kV^{n}}\overset{d}{\longrightarrow}%
\underset{\vphantom{A^{A^A}}\deg\text{ }i_{0}+1}{\prod\nolimits_{n\geq
l}kdV^{n}}\right)  =U_{l}\{-i_{0}\},
\]
and that $H^{j}(M)^{i}=0$ for all other values of $i$ and $j$. Then
$H^{j}(M)^{i}=H^{j}(M)_{p\text{-}\mathrm{tors}}^{i}$, and so $m^{i,j}(M)$ is
zero for all $i$, $j$. The only nonzero $T$ invariant is $T^{i_{0},j_{0}%
}(M)=1$. It follows that the only nonzero Hodge-Witt numbers are%
\[
h_{W}^{i_{0},j_{0}}(M)=1,\quad h_{W}^{i_{0}+1,j_{0}-1}(M)=-2,\quad
h_{W}^{i_{0}+2,j_{0}-2}=1.
\]

\end{example}

\subsection{Weighted Hodge-Witt Euler characteristics}

\begin{theorem}
\label{b5} For every $M$ in $\mathsf{D}_{c}^{b}(R)$ and $r\in\mathbb{Z}{}$,%
\begin{equation}
\sum_{i,j\,\,(i\leq r)}(-1)^{i+j}(r-i)h_{W}^{i,j}(M)=e_{r}(M) \label{eq36}%
\end{equation}
where%
\begin{equation}
e_{r}(M)=\sum_{j}(-1)^{j-1}T^{r-1,j-r}(M)+\sum_{i,j,l\,\,\,(\lambda
_{i,j,l}\leq r-i)}(-1)^{i+j}(r-i-\lambda_{i,j,l}). \label{e26}%
\end{equation}

\end{theorem}

The sum in (\ref{eq36}) is over the pairs of integers $(i,j)$ such that $i\leq
r$, and the first sum in (\ref{e26}) is over the integers $j$. In the second
sum in (\ref{e26}), $(\lambda_{i,j,l})_{l}$ is the family of slopes (with
multiplicities) of the $F$-isocrystal $H^{j}(M)_{K}^{i}$ and the sum is over
the triples $(i,j,l)$ such that $\lambda_{i,j,l}\leq r-i$.

\begin{example}
\label{b4a}Let $M$ be a graded $R$-module, regarded as an element of
$\mathsf{D}_{c}^{b}(R)$ concentrated in degree $j$. Let $F^{\prime}$ act on
$M^{i}$ as $p^{i}F$ (assuming only nonnegative $i$'s occur). Then $F^{\prime}$
is a $\sigma$-linear endomorphism of $M$ regarded as a complex of $R^{0}%
$-modules%
\[
\begin{tikzcd}
\cdots\arrow{r}
&M^{i-1}\arrow{d}{p^{i-1}F}\arrow{r}{d}
&M^{i}\arrow{d}{p^{i}F}\arrow{r}{d}
&M^{i+1}\arrow{d}{p^{i+1}F}\arrow{r}{d}
&\cdots\\
\cdots\arrow{r}
&M^{i-1}\arrow{r}{d}
&M^{i}\arrow{r}{d}
&M^{i+1}\arrow{r}{d}
&\cdots,
\end{tikzcd}
\]
and the second term in (\ref{e26}) equals
\[
\sum\nolimits_{i,l\text{\thinspace\thinspace\thinspace}(\lambda_{i,l}\leq
r)}(-1)^{i+j}(r-\lambda_{i,l})
\]
where $(\lambda_{i,l})_{l}$ is the family of slopes of the $F$-isocrystal
$(M^{i},p^{i}F)_{K}$.
\end{example}

\begin{lemma}
\label{b4}For every distinguished triangle $M^{\prime}\rightarrow M\rightarrow
M^{\prime\prime}\rightarrow M^{\prime}[1]$ in $\mathsf{D}_{c}^{b}(R)$,%
\begin{align*}
m^{i,j}(M) &  =m^{i.j}(M^{\prime})+m^{i,j}(M^{\prime\prime})\\
T^{i,j}(M) &  =T^{i,j}(M^{\prime})+T^{i,j}(M^{\prime\prime}),
\end{align*}
and so%
\[
h_{W}^{i.j}(M)=h_{W}^{i,j}(M^{\prime})+h_{W}^{i,j}(M^{\prime\prime})\text{.}%
\]

\end{lemma}

\begin{proof}
The distinguished triangle gives rise to an exact sequence of graded
$R$-modules%
\[
\cdots\rightarrow H^{j}(M^{\prime})\rightarrow H^{j}(M)\rightarrow
H^{j}(M^{\prime\prime})\rightarrow\cdots.
\]
with only finitely many nonzero terms. It suffices to show that $m$ and $T$
are additive on short exact sequences%
\begin{equation}
0\rightarrow M^{\prime}\rightarrow M\rightarrow M^{\prime\prime}\rightarrow0
\label{eq23}%
\end{equation}
of coherent graded $R$-modules. But $m^{ij}(M)$ depends only on $\bar
{K}\otimes_{W}M$ where $\bar{K}$ is the field of fractions of $\bar{W}$, and
the sequence (\ref{eq23}) splits when tensored with $\bar{K}$. The additivity
of $T$ follows from the description of $T^{i}$ in Lemma \ref{a1a}.
\end{proof}

\begin{lemma}
\label{b4b}For every distinguished triangle $M^{\prime}\rightarrow
M\rightarrow M^{\prime\prime}\rightarrow M^{\prime}[1]$ in $\mathsf{D}_{c}%
^{b}(R)$,%
\begin{equation}
e_{r}(M)=e_{r}(M^{\prime})+e_{r}(M^{\prime\prime})\text{.} \label{eq8}%
\end{equation}

\end{lemma}

\begin{proof}
The same argument as in the proof of Lemma \ref{b4} applies.
\end{proof}

\subsection{Proof of Theorem \ref{b5}}

The numbers do not change under extension of the base field, and so we may
suppose that $k$ is algebraically closed. First note that, if $M^{\prime
}\rightarrow M\rightarrow M^{\prime\prime}\rightarrow M^{\prime}[1]$ is a
distinguished triangle in $\mathsf{D}_{c}^{b}(R)$ and (\ref{eq36}) holds for
$M^{\prime}$ and $M^{\prime\prime}$, then it holds for $M$ (apply \ref{b4} and
\ref{b4b}).

A complex $M$ in $\mathsf{D}_{c}^{b}(R)$ has only finitely many nonzero
cohomology groups, and each has a finite filtration whose quotients are
elementary graded $R$-modules. By using induction on the sum of the lengths of
the shortest such filtrations, one sees that it suffices to prove the formula
for a complex $M$ having only one nonzero cohomology module, which is a degree
shift of an elementary graded $R$-module, i.e., we may assume $M=H^{j_{0}%
}(M)=N\{-i_{0}\}$ where $N$ is elementary.

Assume that $N$ is elementary of type I. If $N$ is torsion, then both sides
are zero. We may suppose that $N$ is a Dieudonne module of slope $\lambda
\in\lbrack0,1[$ with multiplicity $m$ (because $N$ is isogenous to a direct
sum of such modules --- recall that $k$ is algebraically closed). In this case
(see \ref{b2}b), the only nonzero Hodge-Witt invariants of $M$ are%
\begin{align*}
h_{W}^{i_{0},j_{0}}(M)  &  =m^{i_{0},j_{0}}(M)=m(1-\lambda)\\
h_{W}^{i_{0}+1,j_{0}-1}(M)  &  =m^{i_{0}+1,j_{0}-1}(M)=m\lambda.
\end{align*}
Both sides of (\ref{eq36}) are zero if $r\leq i_{0}$, and so we may suppose
that $r>i_{0}$. Then the left hand side (\ref{eq36}) is
\begin{align*}
&  (-1)^{i_{0}+j_{0}}(r-i_{0})h^{i_{0},j_{0}}+(-1)^{i_{0}+1+j_{0}-1}%
(r-i_{0}-1)h^{i_{0}+1,j_{0}-1}\\
=  &  (-1)^{i_{0}+j_{0}}(r-i_{0})(1-\lambda)m+(-1)^{i_{0}+j_{0}}%
(r-i_{0}-1)\lambda m\\
=  &  (-)^{i_{0}+j_{0}}(r-i_{0}-\lambda)m.
\end{align*}
On the other hand, the isocrystal $H^{j}(M)_{K}^{i}$ is zero for
$(i,j)\neq(i_{0},j_{0})$ and $H^{j_{0}}(M)_{K}^{i_{0}}$ is an isocrystal with
slope $\lambda$ of multiplicity $m$, and so%
\[
e_{r}(M)=(-1)^{i_{0}+j_{0}}(r-i_{0}-\lambda)m.
\]

If $N$ is of type II, i.e., $H^{j_{0}}(M)=U_{l}\{-i_{0}\}$, then
$T^{i_{0},j_{0}}=1$ is the only non-zero $T$-invariant (see \ref{b2}c), and so%
\[
e_{r}(M)=\left\{
\begin{array}
[c]{ll}%
(-1)^{i_{0}+j_{0}} & \text{if }r=i_{0}+1\\
0 & \text{otherwise}%
\end{array}
\right.
\]
The nonzero $h_{W}$-invariants are
\[
h_{W}^{i_{0},j_{0}}=1\quad h_{W}^{i_{0}+1,j_{0}-1}=-2\quad h_{W}%
^{i_{0}+2,j_{0}-2}=1,
\]
from which (\ref{eq36}) follows by an elementary calculation.

\begin{aside}
\label{b5c}Here is an alternative proof of Theorem \ref{b5}. Let%
\[
L(r)=\sum\nolimits_{i,j\,\,(i\leq r)}(-1)^{i+j}(r-i)\left(  T^{i,j}%
(M)-2T^{i-1,j+1}(M)+T^{i-2,j+2}(M)\right)  \text{.}%
\]
The contribution of $T^{i_{0},j_{0}}$ to this sum is $(-1)^{i_{0}+j_{0}%
}T^{i_{0},j_{0}}$ if $i_{0}=r-1$ and $0$ otherwise. Therefore%
\begin{equation}
\begin{aligned} L(r) & =\sum\nolimits_{j}(-1)^{r-1+j}T^{r-1,j}\\ & =\sum\nolimits_{j}(-1)^{j-1}T^{r-1,j-r}. \label{eq39} \end{aligned}
\end{equation}

For an $F$-crystal $P$, let $P_{[i,i+1[}=(K\otimes_{W}P)_{[i,i+1[}$ (part with
slopes $\lambda$, $i\leq\lambda<i+1$). From the degeneration of the slope
spectral sequence (\ref{b0}) at $E_{1}$ modulo torsion, we find that%
\[
H^{n}(sM)_{[i,i+1[}\simeq(H^{n-i}(M)_{K}^{i},p^{i}F).
\]
From this, it follows that%
\[
m^{i,n-i}(M)=\sum_{\lambda\in\lbrack i,i+1[}(i+1-\lambda)h_{\lambda}^{n}%
-\sum_{\lambda\in\lbrack i-1,i[}(i-1-\lambda)h_{\lambda}^{n}%
\]
where $h_{\lambda}^{n}$ is the multiplicity of $\lambda$ as a slope of
$H^{n}(sM)$ (cf. \cite{illusie1983}, 6.2). Using these two statements, we find
that%
\begin{equation}
\sum_{i,j\,\,(i\leq r)}(-1)^{i+j}(r-i)m^{i,j}(M)=\sum_{i,j,l\,\,\,(\lambda
_{i,j,l}\leq r-i)}(-1)^{i+j}(r-i-\lambda_{i,j,l}). \label{eq40}%
\end{equation}
On adding (\ref{eq39}) and (\ref{eq40}), we obtain (\ref{eq36}).
\end{aside}

\subsection{Weighted Hodge Euler characteristics}

Following Ekedahl (1986, p.14), we define the \emph{Hodge numbers} of an $M$
in $\mathsf{D}_{c}^{b}(R)$ to be%
\[
h^{i,j}(M)=\dim_{k}(H^{j}(R_{1}\otimes_{R}^{L}M)^{i}).
\]

\begin{theorem}
\label{b5a}For every $M$ in $\mathsf{D}_{c}^{b}(R)$ and $i\in\mathbb{Z}{}$,%
\begin{equation}
\sum\nolimits_{j}(-1)^{j}h_{W}^{i,j}(M)=\sum\nolimits_{j}(-1)^{j}%
h^{i,j}(M)\text{.} \label{eq25}%
\end{equation}

\end{theorem}

\begin{proof}
As for Theorem \ref{b5}, it suffices to prove this for an elementary
$R$-module, where it can be checked directly. See \cite{ekedahl1986}, IV,
Theorem 3.2.
\end{proof}

For $M=R\Gamma(W\Omega_{X}^{\bullet})$, the formula (\ref{eq25}) was found
independently by Crew and Milne (cf. ibid. p.86).

\begin{theorem}
\label{b5b}For every $M$ in $\mathsf{D}_{c}^{b}(R),$%
\begin{equation}
\sum_{i,j\,\,\,(i\leq r)}(-1)^{i+j}(r-i)h^{i,j}(M)=e_{r}(M). \label{eq27}%
\end{equation}

\end{theorem}

\begin{proof}
We have%
\begin{align*}
\text{LHS}  &  \,\,=\sum\nolimits_{i\leq r}(-1)^{i}(r-i)\left(  \sum
\nolimits_{j}(-1)^{j}h^{i,j}(M)\right) \\
&  \overset{\text{(\ref{eq25})}}{=}\sum\nolimits_{i\leq r}(-1)^{i}(r-i)\left(
\sum\nolimits_{j}(-1)^{j}h_{W}^{i,j}(M)\right)  =\text{RHS.}%
\end{align*}

\end{proof}

\section{Internal Homs and tensor products in
$\mathsf{D}_{c}^{b}(R)$}

We review some constructions from \cite{ekedahl1985}.

\subsection{The internal tensor product}

Let $M$ and $N$ be graded $R$-modules. Ekedahl (1985, p.69) defines $M\ast N$
to be the largest quotient of $M\otimes_{W}N$,%
\[
x\otimes y\mapsto x\ast y\colon M\otimes_{W}N\rightarrow M\ast N\text{,}%
\]
in which the following relations hold: $Vx\ast y=V(x\ast Fy)$, $x\ast
Vy=V(Fx\ast y)$, $F(x\ast y)=Fx\ast Fy$, $d(x\ast y)=dx\ast y+(-1)^{\deg
(x)}x\ast dy$.

Regard $W$ as a graded $R$-module concentrated in degree zero with $F$ acting
as $\sigma$. Then
\begin{equation}
W\ast M\simeq M\simeq M\ast W, \label{eq62}%
\end{equation}
and so $W$ plays the role of the identity object $\1$.

The bifunctor $(M,N)\rightsquigarrow M\ast N$ of graded $R$-modules derives to
a bifunctor%
\[
\ast^{L}\colon\mathsf{D}^{-}(R)\times\mathsf{D}^{-}(R)\rightarrow
\mathsf{D}^{-}(R)\text{.}%
\]
If $M$ and $N$ are in $\mathsf{D}_{c}^{b}(R)$, then so also is
\[
M\hat{\ast}N\overset{\textup{{\tiny def}}}{=}\widehat{M\ast^{L}N}\text{.}%
\]
See \cite{ekedahl1985}, I, 4.8; \cite{illusie1983}, 2.6.1.10.

\subsection{The internal Hom}

For graded $R$-modules $M,N$, we let $\Hom^{d}(M,N)$ denote the set of graded
$R$-homomorphisms $M\rightarrow N$ of degree $d$, and we let $\Hom^{\bullet
}(M,N)=\bigoplus_{d}\Hom^{d}(M,N)$. Let $_{R}R$ denote the ring $R$ regarded
as a graded left $R$-module. The internal $\Hom$ of two graded $R$-modules
$M,N$ is
\[
\underline{\Hom}(M,N)\overset{\textup{{\tiny def}}}{=}\Hom^{\bullet}(_{R}R\ast
M,N).
\]
This graded $\mathbb{Z}{}_{p}$-module becomes a graded $R$-module thanks to
the right action of $R$ on $_{R}R$, and $\underline{\Hom}$ derives to a
bifunctor%
\[
R\underline{\Hom}\colon\mathsf{D}(R)^{\mathrm{opp}}\times\mathsf{D}%
^{+}(R)\rightarrow\mathsf{D}(R)
\]
(denoted by $R\underline{\Hom}_{R}$ in \cite{illusie1983}, 2.6.2.6, by
$R\underline{\Hom}_{R}^{!}$ in \cite{ekedahl1985}, p.73, and by $R\Hom_{R}%
^{!}$ in \cite{ekedahl1986}, p.8).

The functor $R\underline{\Hom}(M,N)$ commutes with extension of the base
field. For $M$ in $\mathsf{D}^{-}(R)$ and $N$ in $\mathsf{D}^{+}(R)$,%
\begin{align}
R\underline{\Hom}(W,N)\overset{\text{(\ref{eq62})}}{\simeq}R\Hom^{\bullet
}(_{R}R,N) &  \simeq N\label{eq28c}\\
R\Hom(W,R\underline{\Hom}(M,N)) &  \simeq R\Hom(M,N).\label{eq28}%
\end{align}
(isomorphisms in $\mathsf{D}_{c}^{b}(R)$ and $D(\mathbb{Z}{}_{p})$
respectively). Ekedahl shows that%
\[
R_{1}\otimes_{R}^{L}R\underline{\Hom}(M,N)\simeq R\Hom(R_{1}\otimes_{R}%
^{L}M,R_{1}\otimes_{R}^{L}M)
\]
(isomorphism in $\mathsf{D}(k[d])$) and that%
\begin{equation}
\widehat{R\underline{\Hom}(M,N)}\simeq R\underline{\Hom}(\hat{M},\hat
{N})\text{,}\label{eq41}%
\end{equation}
and so his criterion (see \ref{a6}) shows that $R\underline{\Hom}(M,N)$ lies
in $\mathsf{D}_{c}^{b}(R)$ when both $M$ and $N$ do. See \cite{illusie1983}, 2.6.2.

\section{Homological algebra in the category
$\mathsf{D}_{c}^{b}(R)$}

Throughout this section, $S=\Spec k$, and $\Lambda_{m}=\mathbb{Z}%
/p^{m}\mathbb{Z}{}$.

\subsection{The perfect site}

An $S$-scheme $U$ is \emph{perfect} if its absolute Frobenius map
$F_{\mathrm{abs}}\colon U^{(1/p)}\rightarrow U$ is an isomorphism. The
\emph{perfection} $T^{\mathrm{pf}}$ of an $S$-scheme $T$ is the limit of the
projective system $T\overset{F_{\mathrm{abs}}}{\longleftarrow}T^{(1/p)}%
\overset{F_{\mathrm{abs}}}{\longleftarrow}\cdots$. The scheme $T^{\mathrm{pf}%
}$ is perfect, and for any perfect $S$-scheme $U$, the canonical map
$T^{\mathrm{pf}}\rightarrow T$ defines an isomorphism%
\[
\Hom_{S}(U,T^{\mathrm{pf}})\rightarrow\Hom_{S}(U,T).
\]
Let $\mathsf{Pf}/S$ denote the category of perfect affine schemes over $S$. A
\emph{perfect group scheme} over $S$ is a representable functor $\mathsf{Pf}%
/S\rightarrow\mathsf{Gp}$. For any affine group scheme $G$ over $S$, the
functor $U\rightsquigarrow G(U)\colon\mathsf{Pf}/S\rightarrow\mathsf{Gp}$ is a
perfect group scheme represented by $G^{\mathsf{pf}}$. We say that a perfect
group scheme is \emph{algebraic} if it is represented by an algebraic $S$-scheme.

Let $\mathcal{S}{}$ denote the category of sheaves of commutative groups on
$\left(  \mathsf{Pf}/S\right)  _{\mathrm{et}}$. The commutative perfect
algebraic group schemes killed by some power of $p$ form an abelian
subcategory $\mathcal{G}$ of $\mathcal{S}{}$ which is closed under extensions.
Let $G\in\mathcal{G}$. The identity component $G^{\circ}$ of $G$ has a finite
composition series whose quotients are isomorphic to ${\mathbb{G}}%
_{a}^{\mathrm{pf}}$, and the quotient $G/G{^{\circ}}$ is \'{e}tale. The
dimension of $G$ is the dimension of any algebraic group whose perfection is
$G^{\circ}$. The category $\mathcal{G}{}$ is artinian. See \cite{milne1976},
\S 2, or \cite{berthelot1981}, II.

\begin{example}
\label{n7}Let $f\colon X\rightarrow S$ be a smooth scheme over $S$. The
functors $U\rightsquigarrow\Gamma(U,W_{m}\Omega_{X}^{i})$ are sheaves for the
\'{e}tale topology on $X$. The composite%
\[
W_{m+1}\Omega_{X}^{i}\overset{F}{\longrightarrow}W_{m}\Omega_{X}%
^{i}\rightarrow W_{m}\Omega_{X}^{i}/d(W_{m}\Omega_{X}^{i-1})
\]
factors through $W_{m}\Omega_{X}^{i}$, and so defines a homomorphism%
\[
F\colon W_{m}\Omega_{X}^{i}\rightarrow W_{m}\Omega_{X}^{i}/d(W_{m}\Omega
_{X}^{i-1})\text{.}%
\]
The sheaf $\nu_{m}(i)$ on $X_{\mathrm{et}}$ is defined to be the kernel of%
\[
1-F\colon W_{m}\Omega_{X}^{i}\rightarrow W_{m}\Omega_{X}^{i}/d(W_{m}\Omega
_{X}^{i-1})
\]
(\cite{milne1976}, \S 1; \cite{berthelot1981}, p.209). The map $W_{m+1}%
\Omega_{X}^{i}\rightarrow W_{m}\Omega_{X}^{i}$ defines a surjective map
$\nu_{m+1}(i)\rightarrow\nu_{m}(i)$ with kernel $\nu_{1}(i)$.

Assume that $f$ is proper. The sheaves $R^{i}f_{\ast}\nu_{m}(r)$ lie in
$\mathcal{G}{}$. When $m=1$, this is proved in \cite{milne1976}, 2.7, and the
general case follows by induction on $m$. Following \cite{milne1986}, p.309,
we define%
\begin{align*}
H^{i}(X,(\mathbb{Z}{}/p^{m}\mathbb{Z}{})(r))  &  =H^{i-r}(X_{\mathrm{et}}%
,\nu_{m}(r))\\
H^{i}(X,\mathbb{Z}_{p}(r))  &  =\varprojlim H^{i}(X,(\mathbb{Z}{}%
/p^{m}\mathbb{Z}{})(r))\text{.}%
\end{align*}

\end{example}

\subsection{The functor $M\rightsquigarrow M^{F}$}

For a complex $M$ of graded $R$-modules, we define%
\begin{equation}
M^{F}=R\Hom(W,M). \label{eq55}%
\end{equation}
Then $M\rightsquigarrow M^{F}$ is a functor $\mathsf{D}^{+}(R)\rightarrow
\mathsf{D}(\mathbb{Z}{}_{p})$.

Let $\hat{R}$ denote the completion $\varprojlim R_{m}$ of $R$. From
\[
W\simeq R^{0}/R^{0}(1-F)\simeq R/R(1-F),
\]
we get an exact sequence%
\begin{equation}
0\rightarrow\hat{R}\overset{1-F}{\longrightarrow}\hat{R}\rightarrow
W\rightarrow0 \label{eq20}%
\end{equation}
of graded $R$-modules (\cite{ekedahl1985}, III, 1.5.1, p.90). If $M$ is in
$\mathsf{D}_{c}^{b}(R)$, then, because $M$ is complete,%
\begin{equation}
R\Hom(\hat{R},M)\simeq R\Hom(R,M)\simeq M^{0} \label{eq20a}%
\end{equation}
(isomorphisms of graded $R$-modules; ibid. I, 5.9.3ii, p.78). Now (\ref{eq20})
gives a canonical isomorphism
\begin{equation}
M^{F}\simeq s(M^{0}\overset{1-F}{\longrightarrow}M^{0}) \label{eq22}%
\end{equation}
(ibid. I, 1.5.4(i), p.90), which explains the notation. Note that%

\begin{equation}
s(M^{0}\overset{1-F}{\longrightarrow}M^{0})=\text{\textrm{Cone}}(1-F\colon
M^{0}\rightarrow M^{0})[-1]\text{.} \label{eq22a}%
\end{equation}

For $M,N$ in $\mathsf{D}_{c}^{b}(R)$, we have%
\begin{equation}
R\Hom(M,N)\overset{\text{(\ref{eq28})}}{\simeq}R\Hom(W,R\underline{\Hom}%
(M,N))\overset{\textup{{\tiny def}}}{=}R\underline{\Hom}(M,N)^{F}\label{eq38}%
\end{equation}
in $\mathsf{D}(\mathbb{Z}{}_{p})$.

\subsection{The functor
$M\rightsquigarrow\mathcal{M}_{\bullet}^{F}$}

Let $\mathcal{S}_{\bullet}$ denote the category of projective systems of
sheaves $(P_{m})_{m\in\mathbb{N}}$ on $(\mathsf{Pf}/S)_{\mathrm{et}}$ with
$P_{m}$ a sheaf of $\Lambda_{m}$-modules, and let $\mathcal{G}{}_{\bullet}$
denote the full subcategory of systems $(P_{m})_{m\in\mathbb{N}{}}$ with
$P_{m}$ in $\mathcal{G}{}$. Then $\mathcal{G}{}_{\bullet}$ is an abelian
subcategory of $\mathcal{S}{}_{\bullet}$ closed under extensions.

Let $M$ be a graded $R$-module, and let $M_{m}=R_{m}\otimes_{R}M$. Let
$\mathcal{M}_{m}^{i}{}$ denote the sheaf $\Spec(A)\rightsquigarrow M_{m}%
^{i}\otimes_{W}WA$ on $(\mathsf{Pf}/S)_{\mathrm{e}\mathrm{t}}$, and let
$\mathcal{\mathcal{M}{}}^{i}$ denote the projective system $(\mathcal{M}{}%
_{m}^{i})_{m\in\mathbb{N}{}}$. Thus $\mathcal{M}{}^{i}\in\mathcal{S}%
{}_{\bullet}$. Let $F$ (resp. $V$) denote the endomorphism of $\mathcal{M}%
{}^{i}$ defined by $F\otimes\sigma$ (resp. $V\otimes\sigma^{-1}$) on
$(M_{m}^{i}\otimes_{W}WA)_{m}$. In this way, we get an $R_{\bullet}$-module%
\[
\mathcal{M}{}_{\bullet}\colon\quad\cdots\rightarrow\mathcal{M}{}_{\bullet}%
^{i}\overset{d}{\longrightarrow}\mathcal{M}{}_{\bullet}^{i+1}\rightarrow\cdots
\]
in $\mathcal{S}{}_{\bullet}$. Cf. \cite{illusieR1983} IV, 3.6.3.

\begin{example}
\label{b14}Let $M=M^{0}$ be an elementary graded $R$-module of type I. For
each $m$, the map $1-F\colon\mathcal{M}{}_{m}\rightarrow\mathcal{M}{}_{m}$ is
surjective with kernel the \'{e}tale group scheme $\mathcal{M}{}_{m}^{F}$ over
$k$ corresponding to the natural representation of $\Gal(\bar{k}/k)$ on
$(M\otimes_{W}\bar{W})^{F\otimes\sigma}$ . Therefore $\mathcal{M}{}_{\bullet
}^{F}$ is a pro-\'{e}tale group scheme over $k$ with%
\[
\mathcal{M}{}_{\bullet}^{F}(\bar{k})\overset{\textup{{\tiny def}}%
}{=}\varprojlim\mathcal{M}{}_{m}^{F}(\bar{k})=(M\otimes_{W}\bar{W}%
)^{F\otimes\sigma}.
\]
Cf. (\ref{e4}) below.
\end{example}

\begin{example}
\label{b15}Let $M$ be an elementary graded $R$-module of type II. Then
$1-F\colon\mathcal{M}{}_{\bullet}^{i}\rightarrow\mathcal{M}{}_{\bullet}^{i}$
is bijective for $i=0$, and it is surjective with kernel canonically
isomorphic to $\mathbb{G}{}_{a}^{\mathrm{pf}}$ for $i=1$ (\cite{illusieR1983},
IV, 3.7, p.195).
\end{example}

\begin{proposition}
\label{b6}Let $M$ be a coherent graded $R$-module. For each $i$, the map
$1-F\colon\mathcal{M}{}_{\bullet}^{i}\rightarrow\mathcal{M}{}_{\bullet}^{i}$
is surjective, and its kernel $(\mathcal{M}{}_{\bullet}^{i})^{F}$ lies in
$\mathcal{G}{}_{\bullet}$. There is an exact sequence%
\[
0\rightarrow U^{i}\rightarrow(\mathcal{M}{}_{\bullet}^{F})^{i}\rightarrow
D^{i}\rightarrow0
\]
with $U^{i}$ a connected unipotent perfect algebraic group of dimension
$T^{i-1}(M)$ and $D^{i}$ the profinite \'{e}tale group corresponding to the
natural representation of $\Gal(\bar{k}/k)$ on $(\heartsuit^{i}M\otimes
_{W}\bar{W})^{F\otimes\sigma}$.
\end{proposition}

\begin{proof}
When $M$ is an elementary graded $R$-module, the proposition is proved in the
two examples. The proof can be extended to all coherent graded $R$-modules by
using \cite{illusieR1983}, IV 3.10, 3.11, p.196.
\end{proof}

\begin{corollary}
\label{b6d}Let $M$ be a coherent graded $R$-module, and let $H^{i}%
(M)=Z^{i}(M)/B^{i}(M)$. Then $D^{i}(\bar{k})\overset{\textup{{\tiny def}}%
}{=}\varprojlim_{m}D_{m}^{i}(\bar{k})$ is a finitely generated $\mathbb{Z}%
{}_{p}$-module, and%
\[
D^{i}(\bar{k})\otimes_{\mathbb{Z}{}_{p}}\mathbb{Q}{}_{p}\simeq(H^{i}%
(M)\otimes_{W}\bar{K})^{F\otimes\sigma}\text{.}%
\]

\end{corollary}

\begin{proof}
According to (\ref{b6}), $D^{i}(\bar{k})\simeq(\heartsuit^{i}M\otimes_{W}%
\bar{W})^{F\otimes\sigma}$. Now the statement follows from (\ref{a18}).
\end{proof}

Let $\Gamma(S_{\mathrm{et}},-)$ denote the functor
\[
(M_{m})_{m\in\mathbb{N}{}}\rightsquigarrow\varprojlim\Gamma(S_{\mathrm{et}%
},M_{m})\colon\mathcal{S}{}_{\bullet}\rightarrow\Mod(\mathbb{Z}{}_{p})\text{.}%
\]
It derives to a functor $R\Gamma(S_{\mathrm{et}},-)\colon\mathsf{D}%
(\mathcal{S}{}_{\bullet})\rightarrow\mathsf{D}(\mathbb{Z}{}_{p})$.

For a coherent graded $R$-module $M$, the system $\mathcal{M}_{\bullet}{}$
depends only on the projective system $M_{\bullet}=(M_{m})_{m}$. The functor
$M_{\bullet}\rightsquigarrow\mathcal{M}{}_{\bullet}\colon\Mod(R_{\bullet
})\rightarrow\mathcal{S}{}_{\bullet}$ is exact, and so it defines a functor%
\[
M_{\bullet}\rightsquigarrow\mathcal{M}{}_{\bullet}\colon\mathsf{D}(R_{\bullet
})\rightarrow\mathsf{D}(\mathcal{S}{}_{\bullet}).
\]
Let
\begin{equation}
\mathcal{M}{}_{\bullet}^{F}=\text{\textrm{Cone}}(\mathcal{M}_{\bullet}%
^{0}\overset{1-F}{\longrightarrow}{}\mathcal{M}_{\bullet}^{0})[-1].
\label{eq54}%
\end{equation}

\begin{proposition}
\label{b6e}The following diagram commutes:
\[
\begin{tikzcd}[column sep=large, row sep=large]
\mathsf{D}_{c}^{b}(R)\arrow{r}{M\rightsquigarrow M_{\bullet}}\arrow{rd}[swap]
{M\rightsquigarrow\hat{M}}
& \mathsf{D}^{b}(R_{\bullet})\arrow{r}{(-)^{F}}\arrow{d}{R\varprojlim} & \mathsf{D}^{b}(\mathcal{S}%
{}_{\bullet})\arrow{d}{R\Gamma(S_{\mathrm{et}},-)} \\
& \mathsf{D}^{b}(R)\arrow{r}{(-)^{F}} & \mathsf{D}^{b}(\mathbb{Z}{}_{p}).
\end{tikzcd}
\]
The functor $(-)^{F}$ on the top row (resp. bottom row) is that defined in
(\ref{eq54}) (resp. (\ref{eq55}). In other words, for $M$ in $\mathsf{D}%
_{c}^{b}(R)$,%
\[
R\Gamma(S_{\mathrm{et}},\mathcal{M}{}_{\bullet}^{F})\simeq M^{F}.
\]

\end{proposition}

\begin{proof}
This follows directly from the definitions and the isomorphism (\ref{eq22},
\ref{eq22a})%
\[
M^{F}\simeq\text{\textrm{Cone}}(1-F\colon M^{0}\rightarrow M^{0})[-1]\text{.}%
\]

\end{proof}

\begin{proposition}
\label{b6b}Let $M\in\mathsf{D}_{c}^{b}(R)$, and let $r\in\mathbb{Z}{}$. For
each $j,$ there is an exact sequence%
\[
0\rightarrow U^{j}\rightarrow H^{j}(\mathcal{M(}r){}_{\bullet}^{F})\rightarrow
D^{j}\rightarrow0
\]
with $U^{j}$ a connected unipotent perfect algebraic group of dimension
$T^{r-1,j-r}$ and $D^{j}$ the profinite \'{e}tale group corresponding to the
natural representation of $\Gal(\bar{k}/k)$ on $(\heartsuit^{r}\left(
H^{j}(M)\right)  \otimes\bar{W})^{F\otimes\sigma}$.
\end{proposition}

\begin{proof}
Apply (\ref{b6}) to $H^{j}(M(r))$ with $i=0$.
\end{proof}

\begin{corollary}
\label{b6f}The $\mathbb{Z}{}_{p}$-module $D^{j}(\bar{k})$ is finitely
generated, and%
\begin{equation}
D^{j}(\bar{k})\otimes_{\mathbb{Z}{}_{p}}\mathbb{Q}{}_{p}\simeq(H^{r}%
(H^{j}(M))\otimes\bar{K})^{F\otimes\sigma}\text{.} \label{eq60}%
\end{equation}
Here $H^{r}(H^{j}(M))$ is the $E_{2}^{r,j}$ term in the slope spectral
sequence for $M$.
\end{corollary}

\begin{proof}
Apply (\ref{b6d}) to $H^{j}(M)$.
\end{proof}

\subsection{The functors $R\mathcal{H}om$}

If $M,N$ in $\mathsf{D}_{c}^{b}(R)$, then $P\overset{\textup{{\tiny def}}%
}{=}R\underline{\Hom}(M,N)$ lies in $\mathsf{D}_{c}^{b}(R)$ (see \S 3). Let%
\[
R\mathcal{H}{}om(M,N)=\mathcal{P}{}_{\bullet}^{F}\text{.}%
\]
Then $R\mathcal{H}{}om$ is a bifunctor%
\[
R\mathcal{H}{}om\colon\mathsf{D}_{c}^{b}(R)\times\mathsf{D}_{c}^{b}%
(R)\rightarrow\ \mathsf{D}_{\mathcal{G}{}_{\bullet}}^{b}(\mathcal{S}%
{}_{\bullet})
\]
(denoted by $R\underline{\Hom}_{R}$ in \cite{ekedahl1986}, p.11, except that
he allows graded homomorphisms of any degree).

\begin{proposition}
\label{b6a}For $M,N\in\mathsf{D}_{c}^{b}(R)$,%
\begin{equation}
R\Gamma(S_{\mathrm{et}},R\mathcal{H}om(M,N))\simeq R\Hom(M,N). \label{eq59}%
\end{equation}

\end{proposition}

\begin{proof}
From (\ref{b6e}) with $P=R\underline{\Hom}(M,N)$, we find that%
\[
R\Gamma(S_{\mathrm{et}},R\mathcal{H}om(M,N))\simeq R\underline{\Hom}%
(M,N)^{F}\text{.}%
\]
But $R\underline{\Hom}(M,N)^{F}\simeq R\Hom(M,N)$ (see (\ref{eq38})).
\end{proof}

For $M,N$ in $\mathsf{D}_{c}^{b}(R)$, we let%
\begin{align*}
\Ext^{j}(M,N) &  =H^{j}(R\Hom(M,N))\\
\underline{\Ext}^{j}(M,N) &  =H^{j}(R\underline{\Hom}(M,N))\\
\mathcal{E}{}xt^{j}(M,N) &  =H^{j}(R\mathcal{H}{}om(M,N)).
\end{align*}
The first is a $\mathbb{Z}{}_{p}$-module, the second is a coherent graded
$R$-module, and the third is an object of $\mathcal{G}{}_{\bullet}$. From
(\ref{eq28}) and (\ref{eq59}) we get spectral sequences%
\begin{align*}
\Ext^{i}(W,\underline{\Ext}^{j}(M,N)) &  \implies\Ext^{i+j}(M,N)\\
R^{i}\Gamma(S_{\mathrm{et}},\mathcal{E}{}xt^{j}(M,N)) &  \implies
\Ext^{i+j}(M,N)\text{.}%
\end{align*}
The identity component of $\mathcal{E}{}xt^{j}(M,N)$ is a perfect algebraic
group of dimension $T^{i-1,j}(M,N)$ where%
\[
T^{i,j}(M,N)\overset{\textup{{\tiny def}}}{=}T^{i,j}(R\underline{\Hom}%
(M,N))=T^{i}(\underline{\Ext}^{j}(M,N))\text{.}%
\]
For example, it follows from (\ref{eq28c}) that%
\begin{align*}
\underline{\Ext}^{j}(W,M) &  =H^{j}(M)\\
\mathcal{E}{}xt^{j}(W,M) &  =H^{j}(\mathcal{M}{}_{\bullet}^{F}),\quad
\text{and}\\
T^{i,j}(W,M) &  =T^{i,j}(M).
\end{align*}

\begin{aside}
\label{b6c}If $M\in\mathsf{D}_{c}^{b}(R)$, then the dual%
\[
D(M)\overset{\textup{{\tiny def}}}{=}R\underline{\Hom}(M,W)
\]
of $M$ also lies in $\mathsf{D}_{c}^{b}(R)$. If $M,N\in D_{c}^{b}(R)$, then%
\[
D(M)\hat{\ast}N\simeq R\underline{\Hom}(M,N)
\]
(see \cite{illusie1983}, 2.6.3.4). In particular,%
\[
T^{i,j}(M,N)=T^{i,j}(D(M)\hat{\ast}N).
\]

\end{aside}

\section{The proof of the main theorem}

Throughout this section, $\Gamma$ is a profinite group isomorphic to
$\mathbb{\hat{Z}}{}$, and $\gamma$ is a topological generator for $\Gamma$.
For a $\Gamma$-module $M$, the kernel and cokernel of $1-\gamma\colon
M\rightarrow M$ are denoted by $M^{\Gamma}$ and $M_{\Gamma}$ respectively.

\subsection{Elementary preliminaries}

Let $[S]$ denote the cardinality of a set $S$. For a homomorphism $f\colon
M\rightarrow N$ of abelian groups, we let%
\[
z(f)=\frac{[\Ker(f)]}{[\Coker(f)]}%
\]
when both cardinalities are finite.

\begin{lemma}
\label{a1}Let $M$ be a finitely generated $\mathbb{Z}{}_{p}$-module with an
action of $\Gamma$, and let $f\colon M^{\Gamma}\rightarrow M_{\Gamma}$ be the
map induced by the identity map on $M$. Then $z(f)$ is defined if and only if
$1$ is not a multiple root of the minimum polynomial $\gamma$ on $M$, in which
case $M^{\Gamma}$ has rank equal to the multiplicity of $1$ as an eigenvalue
of $\gamma$ on $M_{\mathbb{Q}{}_{p}}$ and%
\[
z(f)=\left\vert \prod\nolimits_{i,\,\,a_{i}\neq1}(1-a_{i})\right\vert _{p}%
\]
where $\left(  a_{i}\right)  _{i\in I}$ is the family of eigenvalues of
$\gamma$ on $M_{\mathbb{Q}{}_{p}}$.
\end{lemma}

\begin{proof}
Elementary and easy.\noindent
\end{proof}

\begin{lemma}
\label{b10}Consider a commutative diagram%
\[
\begin{CD}
\cdots@>>>C^{j-1}@.C^{j}@>{f^{j}}>>C^{j+1}@.\\
@.@VV{g^{j-1}}V@AA{h^{j}}A@VV{g^{j+1}}V\\
\cdots @>>> A^{j-1} @>{d^{j-1}}>> A^{j} @>{d^{j}}>> A^{j+1} @>>> \cdots\\
@.@VV{h^{j-1}}V@AA{g^{j}}A@VV{h^{j+1}}V\\
@.B^{j-1}@>{f^{j-1}}>>B^{j}@.B^{j+1}@>>>\cdots
\end{CD}
\]
in which $A^{\bullet}$ is a bounded complex of abelian groups and each column
is a short exact sequence (in particular, the $g$'s are injective and the
$h$'s are surjective). The cohomology groups $H^{j}(A^{\bullet})$ are all
finite if and only if the numbers $z(f^{j})$ are all defined, in which case%
\[
\prod\nolimits_{j}[H^{j}(A^{\bullet})]^{(-1)^{j}}=\prod\nolimits_{j}%
z(f^{j})^{(-1)^{j}}.
\]

\end{lemma}

\begin{proof}
Because $h^{j-1}$ is surjective, $g^{j}$ maps the image of $f^{j-1}$ into the
image of $d^{j-1}$. Because $g^{j+1}$ is injective and $h^{j}$ is surjective,
$h^{j}$ maps the kernel of $d^{j}$ onto the kernel of $f^{j}$. The snake lemma
applied to%
\[
\begin{CD}
@.\im(f^{j-1})@>{g^{j}}>>\im(d^{j-1}) @>>> 0 \\
@.@VVV@VVV@VVV@.\\
0 @>>> B^{j} @>g^{j}>> \Ker(d^{j}) @>{h^j}>> \Ker(f^{j}) @>>> 0
\end{CD}
\]
gives an exact sequence%
\[
0\rightarrow\Coker(f^{j-1})\rightarrow H^{j}(A^{\bullet})\rightarrow
\Ker(f^{j})\rightarrow0.
\]
Therefore $H^{j}(A^{\bullet})$ is finite if and only if $\Coker(f^{j-1})$ and
$\Ker(f^{j})$ are both finite, in which case%
\[
\lbrack H^{j}(A^{\bullet})]=[\Coker(f^{j-1})]\cdot\lbrack\Ker(f^{j})].
\]
On combining these statements for all $j$, we obtain the lemma.
\end{proof}

\subsection{Cohomological preliminaries}

Let $\Lambda$ be a finite ring, and let $\Lambda\Gamma$ be the group ring. For
a $\Lambda$-module $M$, we let $M_{\ast}$ denote the corresponding co-induced
module. Thus $M_{\ast}$ consists of the locally constant maps $f\colon
\Gamma\rightarrow M$ and $\tau\in\Gamma$ acts on $f$ according to the rule
$(\tau f)(\sigma)=f(\sigma\tau)$. When $M$ is a discrete $\Gamma$-module,
there is an exact sequence%
\begin{equation}
0\rightarrow M\longrightarrow M_{\ast}\overset{\alpha_{\gamma}%
}{\longrightarrow}M_{\ast}\rightarrow0, \label{eq43}%
\end{equation}
in which the first map sends $m\in M$ to the map $\sigma\mapsto\sigma m$ and
the second map sends $f\in M_{\ast}$ to $\sigma\mapsto f(\sigma\gamma)-\gamma
f(\sigma)$. Let $F$ be the functor $M\rightsquigarrow M^{\Gamma}%
\colon\Mod(\Lambda\Gamma)\rightarrow\Mod(\Lambda)$. The class of co-induced
$\Lambda\Gamma$-modules is $F$-injective, and so (\ref{eq43}) defines
isomorphisms
\[
RF(M)\simeq F(M_{\ast}\overset{\alpha_{\gamma}}{\longrightarrow}M_{\ast
})\simeq(M\overset{1-\gamma}{\longrightarrow}M)
\]
in $\mathsf{D}^{+}(\Lambda)$. For the second isomorphism, note that $M_{\ast
}^{\Gamma}$ is the set of constant functions $\Gamma\rightarrow M$, and if $f$
is the constant function with value $m$, then $(\alpha_{\gamma}f)(\sigma
)=f(\sigma\gamma)-\gamma f(\sigma)=m-\gamma m$.

Now let $\Mod(\Lambda_{\bullet}\Gamma)$ denote the category of projective
systems $(M_{m})_{m\in\mathbb{N}{}}$ with $M_{m}$ a discrete $\Gamma$-module
killed by $p^{m}$, and let $F$ be the functor $\Mod(\Lambda_{\bullet}%
\Gamma)\rightarrow\Mod(\mathbb{Z}_{p})$ sending $(M_{m})_{m}$ to $\varprojlim
M_{m}^{\Gamma}$. We say that an object $(M_{m})_{m}$ of $\Mod(\Lambda
_{\bullet}\Gamma)$ is co-induced if $M_{m}$ is co-induced for each $m$. For
every complex $X=(X_{m})_{m}$ of $\Lambda_{\bullet}\Gamma$-modules, there is
an exact sequence%
\begin{equation}
0\rightarrow X\rightarrow X_{\ast}\overset{\alpha_{\gamma}}{\longrightarrow
}X_{\ast}\rightarrow0\label{eq46}%
\end{equation}
of complexes with $X_{\ast}^{j}=(X_{m\ast}^{j})_{m}$ for all $j,m$. The class
of co-induced $\Lambda_{\bullet}\Gamma$-modules is $F$-injective, and so
(\ref{eq46}) defines isomorphisms%
\begin{equation}
RF(X)\simeq s(F(X_{\ast}\longrightarrow X_{\ast}))\simeq s(\vec{X}%
\overset{1-\gamma}{\longrightarrow}\vec{X})\label{eq47}%
\end{equation}
in $\mathsf{D}^{+}(\mathbb{Z}{}_{p})$ where $\vec{X}=(R\varprojlim)(X)$ and
$\vec{X}\overset{1-\gamma}{\longrightarrow}\vec{X}$ is a double complex with
$\vec{X}$ as both its zeroth and first column. From (\ref{eq47}), we get a
long exact sequence%
\begin{equation}
\cdots\rightarrow H^{j-1}(\vec{X})\overset{1-\gamma}{\longrightarrow}%
H^{j-1}(\vec{X})\rightarrow R^{j}F(X)\rightarrow H^{j}(\vec{X}%
)\overset{1-\gamma}{\longrightarrow}H^{j}(\vec{X})\rightarrow\cdots
.\label{eq48}%
\end{equation}

If $(M_{m})_{m}$ is a $\Lambda_{\bullet}\Gamma$-module satisfying the
Mittag-Leffler condition, then
\[
R^{j}F((M_{m})_{m})\simeq H_{\mathrm{cts}}^{j}(\Gamma,\varprojlim M_{m})
\]
(continuous cohomology). Let $\Lambda_{\bullet}=(\mathbb{Z}{}/p^{m}%
\mathbb{Z}{})_{m}$. Then%
\[
R^{1}F(\Lambda_{\bullet})\simeq H_{\mathrm{cts}}^{1}(\Gamma,\mathbb{Z}{}%
_{p})\simeq\Hom_{\mathrm{cts}}(\Gamma,\mathbb{Z}{}_{p})\text{,}%
\]
which has a canonical element $\theta$, namely, that mapping $\gamma$ to $1$.
We can regard $\theta$ as an element of%
\[
\Ext^{1}(\Lambda_{\bullet},\Lambda_{\bullet})=\Hom_{\mathsf{D}^{+}%
(\Lambda_{\bullet}\Gamma)}(\Lambda_{\bullet},\Lambda_{\bullet}[1]).
\]
Thus, for $X$ in $\mathsf{D}^{+}(\Lambda_{\bullet}\Gamma)$, we obtain maps%
\begin{align*}
\theta\colon X  &  \rightarrow X[1]\\
R\theta\colon RF(X)  &  \rightarrow RF(X)[1].
\end{align*}
The second map is described explicitly by the following map of double
complexes:%
\[
\begin{tikzpicture}[baseline=(current bounding box.center), text height=1.5ex, text depth=0.25ex]
\node (a) at (0,0) {$RF(X)$};
\node (b) at (2,0) {};
\node (c) at (4,0) {$\vec{X}$};
\node (d) at (6,0) {$\vec{X}$};
\node (e) [below=of a] {$RF(X)[1]$};
\node (f)[below=of b] {$\vec{X}$};
\node (g)[below=of c] {$\vec{X}$};
\node (h)[below=of d] {};
\node at (2,-2.05) {$\scriptstyle{-1}$};
\node at (4,-2.05) {$\scriptstyle{0}$};
\node at (6,-2.05) {$\scriptstyle{1}$};
\draw[->,font=\scriptsize,>=angle 90]
(a) edge node[right]{$R\theta$} (e)
(c) edge node[right]{$\gamma$} (g)
(c) edge node[above]{$1-\gamma$} (d)
(f) edge node[above]{$1-\gamma$} (g);
\end{tikzpicture}
\]
For all $j$, the following diagram commutes%
\begin{equation}
\begin{tikzcd} R^{j}F(X)\arrow{d}\arrow{r}{d^{j}}& R^{j+1}F(X)\\ H^{j}(\vec{X})\arrow{r}{\id} & H^{j}(\vec{X})\arrow{u} \end{tikzcd} \label{eq50}%
\end{equation}
where $d^{j}=H^{j}(R\theta)$ and the vertical maps are those in (\ref{eq48}).
The sequence%
\begin{equation}
\cdots\longrightarrow R^{j-1}F(X)\overset{d^{j-1}}{\longrightarrow}%
R^{j}F(X)\overset{d^{j}}{\longrightarrow}R^{j+1}F(X)\longrightarrow
\cdots\label{eq51}%
\end{equation}
is a complex because $R\theta\circ R\theta=0$.

\subsection{Review of $F$-isocrystals}

Let $V$ be an $F$-isocrystal over $k$. The $\bar{K}$-module $\bar
{V}\overset{\textup{{\tiny def}}}{=}\bar{K}\otimes_{K}V$ becomes an
$F$-isocrystal over $\bar{k}$ with $\bar{F}$ acting as $\sigma\otimes F$.

\begin{plain}
\label{c2} Let $\lambda$ be a nonnegative rational number, and write
$\lambda=s/r$ with $r,s\in\mathbb{N}{}$, $r>0$, $(r,s)=1$. Define $E^{\lambda
}$ to be the $F$--isocrystal $K_{\sigma}[F]/\left(  K_{\sigma}[F](F^{r}%
-p^{s})\right)  $.

When $k$ is algebraically closed, every $F$-isocrystal is semisimple, and the
simple $F$-isocrystals are exactly the $E^{\lambda}$ with $\lambda
\in\mathbb{Q}_{\geq0}$. Therefore an $F$-isocrystal has a unique (slope)
decomposition%
\begin{equation}
V=\bigoplus\nolimits_{\lambda\geq0}V_{\lambda} \label{eq63}%
\end{equation}
with $V_{\lambda}$ a sum of copies of $E^{\lambda}$. See \cite{demazure1972}, IV.

When $k$ is merely perfect, the decomposition (\ref{eq63}) of $\bar{V}$ is
stable under $\Gal(\bar{k}/k)$, and so arises from a (slope) decomposition of
$V$. In other words, $V=\bigoplus\nolimits_{\lambda}V_{\lambda}$ with
$\overline{V_{\lambda}}=\bar{V}_{\lambda}$. If $\lambda=r/s$ with $r,s$ as
above, then $V_{\lambda}$ is the largest $K$-submodule of $V$ such that
$F^{r}V_{\lambda}=p^{s}V_{\lambda}$. The $F$-isocrystal $V_{\lambda}$ is
called the part of $V$ with slope $\lambda$, and $\{\lambda\mid V_{\lambda
}\neq0\}$ is the set of slopes of $V$.
\end{plain}

\begin{plain}
\label{e3}Let $V$ be an $F$-isocrystal over $k$. The characteristic
polynomial
\[
P_{V,\alpha}(t)\overset{\textup{{\tiny def}}}{=}\det(1-\alpha t|V)
\]
of an endomorphism $\alpha$ of $V$ lies in $\mathbb{Q}_{p}[t]$. Let
$k=\mathbb{F}{}_{q}$ with $q=p^{a}$, so that $F^{a}$ is an endomorphism of
$(V,F)$, and let
\[
P_{V,F^{a}}(t)=\prod\nolimits_{i\in I}(1-a_{i}t),\quad a_{i}\in\mathbb{\bar
{Q}}_{p}\text{.}%
\]
According to a theorem of Manin, $\left(  \ord_{q}(a_{i})\right)  _{i\in I}$
is the family of slopes of $V$. Here $\ord_{q}$ is the $p$-adic valuation on
$\mathbb{\bar{Q}}_{p}$ normalized so that $\ord_{q}(q)=1$. See
\cite{demazure1972}, pp.89-90.
\end{plain}

\begin{plain}
\label{e4}Let $(V,F)$ be an $F$-isocrystal over $k=\mathbb{F}{}_{p^{a}}$, and
let $\lambda\in\mathbb{N}{}$. Let%
\[
V_{(\lambda)}=\{v\in\bar{V}\mid\bar{F}v=p^{\lambda}v\}\quad\text{(}%
\mathbb{Q}{}_{p}\text{-subspace of }\bar{V}\text{).}%
\]
Then $V_{(\lambda)}$ is a $\mathbb{Q}{}_{p}$-structure on $V_{\lambda}$. In
other words, $V_{\lambda}$ has a basis of elements $e$ with the property that
$\bar{F}e=p^{\lambda}e$, and hence
\[
\left(  \gamma\otimes F^{a}\right)  e=\bar{F}^{a}e=q^{\lambda}e.
\]
Therefore, as $c$ runs over the eigenvalues of $F^{a}$ on $V$ with
$\ord_{q}(c)=\lambda$, the quotient $q^{\lambda}/c$ runs over the eigenvalues
of $\gamma$ on $V_{(\lambda)}$; moreover, $c$ is a multiple root of the
minimum polynomial of $F^{a}$ on $V_{\lambda}$ if and only if $q^{\lambda}/c$
is a multiple root of the minimum polyomial of $\gamma$ on $V_{(\lambda)}$.
See \cite{milne1986}, 5.3.
\end{plain}

\begin{plain}
\label{e5}Let $(V,F)$ be an $F$-isocrystal over $k=\mathbb{F}{}_{p^{a}}$. If
$F^{a}$ is a semisimple endomorphism of $V$ (as a $K$-vector space), then
$\End(V,F)$ is semisimple, because it is a $\mathbb{Q}{}_{p}$-form of the
centralizer of $F^{a}$ in $\End(V)$; it follows that $(V,F)$ is semisimple.
Conversely, if $(V,F)$ is semisimple, then $F^{a}$ is semisimple, because it
lies in the centre of the semisimple algebra $\End(V,F)$. Let $V$ and
$V^{\prime}$ be nonzero $F$-isocrystals; then $V\otimes V^{\prime}$ is
semisimple if and only if both $V$ and $V^{\prime}$ are semisimple.
\end{plain}

\subsection{A preliminary calculation}

In this subsection, $k$ is the finite field $\mathbb{F}_{q}$ with $q=p^{a}$,
and $\Gamma=\Gal(\bar{k}/k)$. We take the Frobenius element $x\mapsto x^{q}$
to be the generator $\gamma$ of $\Gamma$.

Recall that for $P$ in $\mathsf{D}_{c}^{b}(R)$, $H^{j}(sP)_{K}$ is an $F$-isocrystal.

\begin{proposition}
\label{b9c}Let $M,N\in\mathsf{D}_{c}^{b}(R)$, let $P=R\underline{\Hom}(M,N)$,
and let $r\in\mathbb{Z}{}$. For each $j$, let
\[
f_{j}\colon\Ext^{j}(\bar{M},\bar{N}(r))^{\Gamma}\rightarrow\Ext^{j}(\bar
{M},\bar{N}(r))_{\Gamma}%
\]
be the map induced by the identity map. Then $z(f_{j})$ is defined if and only
if $q^{r}$ is not a multiple root of the minimum polynomial of $F^{a}$ on
$H^{j}(sP)_{K}$, in which case%
\[
z(f_{j})=\left\vert \prod_{a_{j,l}\neq q^{r}}\left(  1-\frac{a_{j,l}}{q^{r}%
}\right)  \right\vert _{p}~~\left\vert \prod_{\mathrm{ord}_{q}(a_{j,l}%
)<r}\frac{q^{r}}{a_{j,l}}\right\vert _{p}q^{T^{r-1,j-r}(P)}%
\]
where $(a_{j,l})_{l}$ is the family of eigenvalues of $F^{a}$ acting on
$H^{j}(sP)_{K}$.
\end{proposition}

\begin{proof}
(Following the proof of \cite{milne1986}, 6.2.) Let $G^{j}$ denote the perfect
pro-group scheme $\mathcal{E}{}xt^{j}(M,N(r))\overset{\textup{{\tiny def}}%
}{=}H^{j}(\mathcal{P}(r){}_{\bullet}^{F})$. There is an exact sequence%
\[
0\rightarrow U^{j}\rightarrow G^{j}\rightarrow D^{j}\rightarrow0
\]
in which $U^{j}$ is a connected unipotent perfect algebraic group of dimension
$T^{r-1,j-r}(P)$ and $D^{j}$ is a pro-\'{e}tale group such that $D^{j}(\bar
{k})$ is a finitely generated $\mathbb{Z}{}_{p}$-module and%
\begin{equation}
D^{j}(\bar{k})\otimes_{\mathbb{Z}{}_{p}}\mathbb{Q}{}_{p}\simeq H^{j}%
(sP(r)_{K})_{(0)}\simeq H^{j}(sP_{K})_{(r)} \label{eq61}%
\end{equation}
(see \ref{b6b}, \ref{b6f}). The map $1-\gamma\colon U^{j}(\bar{k})\rightarrow
U^{j}(\bar{k})$ is surjective because it is \'{e}tale and $U^{j}$ is
connected. On applying the snake lemma to%
\[
\begin{CD}
0 @>>> U^{j}(\bar{k}) @>>> G^{j}(\bar{k}) @>>> D^{j}(\bar{k}) @>>> 0\\
@.@VV{1-\gamma}V @VV{1-\gamma}V @VV{1-\gamma}V\\
0@>>>U^{j}(\bar{k})@>>> G^{j}(\bar{k}) @>>> D^{j}(\bar{k}) @>>> 0,
\end{CD}
\]
and using that the first vertical arrow is surjective, we obtain the upper and
lower rows of the following exact commutative diagram%
\begin{equation}
\begin{CD} 0 @>>> U^j(\bar{k})^{\Gamma} @>>> G^j (\bar{k})^{\Gamma} @>>> D^j (\bar{k})^{\Gamma} @>>> 0\\ @. @VV{f^{\prime}_j}V @VV{f_j}V @VV{f^{\prime\prime}_j}V @.\\ 0 @>>> 0 @>>> G^j (\bar{k})_{\Gamma} @>>> D^j (\bar{k})_{\Gamma} @>>> 0. \end{CD} \label{eq2a}%
\end{equation}
Because $U^{j}$ has a composition series whose quotients are isomorphic to
$\mathbb{G}{}_{a}^{\mathrm{pf}}$,%
\[
\lbrack U^{j}(k)]=q^{\dim(U^{j})}=q^{T^{r-1,j-r}}\text{.}%
\]
On the other hand, it follows from (\ref{e4}) that the eigenvalues of $\gamma$
acting on $D^{j}(\bar{k})_{\mathbb{Q}{}_{p}}$ are the quotients $q^{r}%
/a_{j,l}$ with $\ord_{q}(a_{j,l})=r$. Therefore, (\ref{a1}) and (\ref{e4})
show that $z(f_{j}^{\prime\prime})$ is defined if and only if the minimum
polynomial of $F^{a}$ on $H^{j}(sP)_{K}$ does not have $q^{r}$ as a multiple
root, in which case%
\[
z(f_{j}^{\prime\prime})=\left\vert \prod_{l}\left(  1-\frac{q^{r}}{a_{j,l}%
}\right)  \right\vert _{p}%
\]
where the product is over the $a_{j,l}$ such that $\ord_{q}(a_{j,l})=r$ but
$a_{j,l}\neq q^{r}$. Note that%
\[
\left\vert 1-\frac{a_{j,l}}{q^{r}}\right\vert _{p}=\left\vert 1-\frac{q^{r}%
}{a_{j,l}}\right\vert _{p}\quad\text{if }\ord_{q}(a_{j,l})=r,
\]
and%
\[
\left\vert 1-\frac{a_{j,l}}{q^{r}}\right\vert _{p}=\left\{
\begin{array}
[c]{ll}%
|a_{j,l}/q^{r}|_{p} & \text{if }\ord_{q}(a_{j,l})<r\\
1 & \text{if }\ord_{q}(a_{j,l})>r.
\end{array}
\right.
\]
Therefore%
\[
z(f_{j}^{\prime\prime})=\left\vert \prod_{a_{j,l}\neq q^{r}}\left(
1-\frac{a_{j,l}}{q^{r}}\right)  \right\vert _{p}~~\left\vert \prod
_{\mathrm{ord}_{q}(a_{j,l})<r}\frac{q^{r}}{a_{j,l}}\right\vert _{p}%
\]
where both products are over all $a_{j,l}$ satisfying the conditions. The
snake lemma applied to (\ref{eq2a}) shows $z(f_{j})$ is defined if and only if
both $z(f_{j}^{\prime})$ and $z(f_{j}^{\prime\prime})$ are defined, in which
case $z(f_{j})=z(f_{j}^{\prime})\cdot z(f_{j}^{\prime\prime})$. The
proposition now follows.
\end{proof}

\subsection{Definition of the complex $E(M,N(r))$}

Recall (\ref{b6a}) that the bifunctor%
\[
R\Hom\colon\mathsf{D}(R)^{\mathrm{opp}}\times\mathsf{D}^{+}(R)\rightarrow
\mathsf{D}(\mathbb{Z}{}_{p})
\]
factors canonically through%
\[
R\Gamma(S_{\mathrm{et}},-)\colon\mathsf{D}^{+}(\mathcal{S}{}_{\bullet
})\rightarrow\mathsf{D}(\mathbb{Z}{}_{p})
\]
where $\Gamma(S_{\mathrm{et}},-)$ is the functor $(P_{m})_{m}\rightsquigarrow
\varprojlim\Gamma(S_{\mathrm{et}},P_{m})$. Since $R\Gamma(S_{\mathrm{et}},-)$
obviously factors through%
\[
RF\colon\mathsf{D}^{+}(\Lambda_{\bullet}\Gamma)\rightarrow\mathsf{D}%
(\mathbb{Z}{}_{p}),\quad F=\left(  (M_{m})_{m}\rightsquigarrow\varprojlim
M_{m}^{\Gamma}\right)  ,\quad\Gamma=\Gal(\bar{k}/k),
\]
so also does $R\Hom$. Therefore, for $M,N\in\mathsf{D}_{c}^{b}(R)$, there
exists a well-defined object $X$ in $\mathsf{D}^{+}(\Lambda_{\bullet}\Gamma)$
such that $RF(X)=R\Hom(M,N(r))$. For an algebraically closed base field $k$,
$RF(X)=\vec{X}$, and so, for a general $k$, $\vec{X}=R\Hom(\bar{M},\bar
{N}(r))$.

Now let $k$ be $\mathbb{F}{}_{q}$ with $q=p^{a}$. With $X$ as in the last
paragraph, the sequence (\ref{eq48}) gives us short exact sequences%
\begin{equation}
0\rightarrow\Ext^{j-1}(\bar{M},\bar{N}(r))_{\Gamma}\rightarrow\Ext^{j}%
(M,N(r))\rightarrow\Ext^{j}(\bar{M},\bar{N}(r))^{\Gamma}\rightarrow0\text{.}
\label{eq52}%
\end{equation}
Moreover, (\ref{eq51}) becomes a complex%
\[
E(M,N(r))\colon\quad\cdots\rightarrow\Ext^{j-1}(M,N(r))\rightarrow
\Ext^{j}(M,N(r))\rightarrow\Ext^{j+1}(M,N(r))\rightarrow\cdots
\]
This is the unique complex for which the following diagram commutes,%
\begin{equation}
\minCDarrowwidth5pt\begin{CD} @.@.\Ext^j(\bar{M},\bar{N}(r))^{\Gamma}@>f^j>> \Ext^{j}(\bar{M},\bar{N}(r))_{\Gamma}@.\\ @.@.@AAA@VVV\\ \cdots @>>> \Ext^{j-1}(M,N(r)) @>{d^{j-1}}>> \Ext^{j}(M,N(r)) @>{d^{j}}>> \Ext^{j+1}(M,N(r)) @>>> \cdots\\ @.@VVV@AAA@.\\ @.\Ext^{j-1}(\bar{M},\bar{N}(r))^{\Gamma}@>f^{j-1}>>\Ext^{j-1}(\bar{M},\bar{N}(r))_{\Gamma}@.@. \end{CD} \label{e30}%
\end{equation}
(the vertical maps are those in (\ref{eq52}) and the maps $f^{j}$ are induced
by the identity map).

Let $P\in\mathsf{D}_{c}^{b}(R)$. The zeta function $Z(P,t)$ of $P$ is the
alternating product of the characteristic polynomials of $F^{a}$ acting on the
isocrystals $H^{j}(sP)_{K}$:
\[
Z(P,t)=\prod\nolimits_{j}\det(1-F^{a}t\mid H^{j}(sP)_{K})^{(-1)^{j+1}}.\quad
\]

\subsection{Proof of Theorem \ref{a0}}

We first note that the condition on the minimum polynomial of $F^{a}$ implies
that the minimum polynomial of $\gamma$ on $H^{j}(sP_{\bar{K}})_{(r)}$ does
not have $1$ as a multiple root (see \ref{e4}). Let%
\[
P_{j}(t)=\det(1-F^{a}t\mid H^{j}(sP)_{K})=\prod\nolimits_{l}(1-a_{j,l}t).
\]

(a) We have $\Ext^{j}(\bar{M},\bar{N}(r))=\mathcal{E}{}xt^{j}(M,N(r))(\bar
{k})$ where $\mathcal{E}{}xt^{j}(M,N(r))$ is a pro-algebraic group such that
the identity component of $\mathcal{E}{}xt^{j}(M,N(r))$ is algebraic and the
quotient of $\mathcal{E}{}xt^{j}(M,N(r))$ by its identity component is a
pro-\'{e}tale group $(D_{m}^{j})_{m}$ such that $\varprojlim_{m}D_{m}^{j}%
(\bar{k})$ is a finitely generated $\mathbb{Z}{}_{p}$-module (see \ref{b6b},
\ref{b6f}). Hence the $\mathbb{Z}{}_{p}$-modules $\Ext^{j}(\bar{M},\bar
{N}(r))^{\Gamma}$ and $\Ext^{j}(\bar{M},\bar{N}(r))_{\Gamma}$ are finitely
generated. Now%
\[
\rank(\Ext^{j}(M,N(r)))=\rank(\Ext^{j-1}(\bar{M},\bar{N}(r))_{\Gamma
})+\rank(\Ext^{j}(\bar{M},\bar{N}(r))^{\Gamma}).
\]
The hypothesis on the action of the Frobenius element implies that%
\[
\Ext^{j}(\bar{M},\bar{N}(r))^{\Gamma}\otimes\mathbb{Q}{}\simeq\Ext^{j}(\bar
{M},\bar{N}(r))_{\Gamma}\otimes\mathbb{Q}{}%
\]
for all $j$, and so%
\[
\rank(\Ext^{j}(M,N(r)))=\rank(\Ext^{j-1}(\bar{M},\bar{N}(r))^{\Gamma
})+\rank(\Ext^{j}(\bar{M},\bar{N}(r))^{\Gamma}).
\]
Therefore,%
\[
\sum\nolimits_{j}(-1)^{j}\rank(\Ext^{j}(M,N(r)))=0.
\]
(b) Let $\rho_{j}$ be the multiplicity of $q^{r}$ as an inverse root of
$P_{j}$. Then%
\[
\rho_{j}=\rank\Ext^{j}(\bar{M},\bar{N}(r))^{\Gamma}=\rank\Ext^{j}(\bar{M}%
,\bar{N}(r))_{\Gamma}\text{,}%
\]
and so%
\begin{align*}
\sum\nolimits_{j}(-1)^{j+1}\cdot j\cdot\rank(\Ext^{j}(M,N(r)))  &
=\sum\nolimits_{j}(-1)^{j+1}\cdot j\cdot(\rho_{j-1}+\rho_{j})\\
&  =\sum\nolimits_{j}(-1)^{j}\rho_{j}\\
&  =\rho\text{.}%
\end{align*}

(c) From Lemma \ref{b10} applied to the diagram (\ref{e30}), we find that%
\[
\chi(M,N(r))=\prod\nolimits_{j}z(f^{j})^{(-1)^{j}}\text{.}%
\]
According to Proposition \ref{b9c},%
\[
z(f^{j})=\left\vert \prod_{a_{j,l}\neq q^{r}}\left(  1-\frac{a_{j,l}}{q^{r}%
}\right)  \right\vert _{p}~~\left\vert \prod_{\mathrm{ord}_{q}(a_{j,l}%
)<r}\frac{q^{r}}{a_{j,l}}\right\vert _{p}q^{T^{r-1,j-r}(P).}.
\]
where $(a_{j,l})_{l}$ is the family of eigenvalues of $F^{a}$ acting on
$H^{j}(sP(r))_{\mathbb{Q}{}}$. Note that%
\[
\prod_{a_{j,l}\neq q^{r}}\left(  1-\frac{a_{j,l}}{q^{r}}\right)
=\lim_{t\rightarrow q^{-r}}\frac{P_{j}(t)}{(1-q^{r}t)^{\rho_{j}}}\text{.}%
\]
According to (\ref{e3}),%
\[
\left\vert \prod_{\mathrm{ord}_{q}(a_{j,l})<r}\frac{q^{r}}{a_{j,l}}\right\vert
_{p}^{-1}=\sum_{l\,\,\,(\lambda_{j,l}<r)}r-\lambda_{j,l}%
\]
where $(\lambda_{j,l})_{l}$ is the family of slopes $H^{j}(sP(r))_{\mathbb{Q}%
{}}$. Therefore%
\[
\chi(M,N(r))=\left\vert \lim_{t\rightarrow q^{-r}}Z(M,N,t)\cdot(1-q^{r}%
t)^{\rho}\right\vert _{p}^{-1}q^{-e_{r}(P)}.
\]
Theorem \ref{b5b} completes the proof.

\section{Applications to algebraic varieties}

Throughout, $S=\Spec(k)$ where $k$ is perfect field of characteristic $p>0$.

Recall that the zeta function of an algebraic variety $X$ over a finite field
$\mathbb{F}{}_{q}$ is defined to be the formal power series $Z(X,t)\in
\mathbb{Q}{}[[t]]$ such that%
\begin{equation}
\log(Z(X,t))=\sum_{n>0}\frac{N_{n}t_{n}}{n},\quad N_{n}=\#(X(\mathbb{F}%
{}_{q^{n}})), \label{eq3}%
\end{equation}
and that Dwork (1960)\nocite{dwork1960} proved that $Z(X,t)\in\mathbb{Q}{}(t)$.

\subsection{Smooth complete varieties}

Let $X$ be a smooth complete variety over a perfect field $k$, and let%
\[
M(X)=R\Gamma(X,W\Omega_{X}^{\bullet})\in\mathsf{D}_{c}^{b}(R)
\]
(see \ref{b1}). For all $j\geq0$,%
\begin{equation}
H^{j}(s\left(  M(X)\right)  )\simeq H_{\mathrm{crys}}^{j}(X/W) \label{eq65}%
\end{equation}
(isomorphism of $F$-isocrystals; see \ref{b1}), and so%
\[
Z(M(X),t)=\prod\nolimits_{j}\det(1-F^{a}t\mid H_{\mathrm{crys}}^{j}%
(X/W)_{\mathbb{Q}{}}^{(-1)^{j+1}}.
\]
That this equals $Z(X,t)$ is proved in \cite{katzM1974} for $X$ projective,
and the complete case can be deduced from the projective case by using de
Jong's theory of alterations (\cite{suh2012}). Moreover, $H_{\mathrm{crys}%
}^{j}(X/W)_{\mathbb{Q}{}}$ can be replaced by $H_{\mathrm{rig}}^{j}(X)$ (see
\ref{r2} below). Finally, $H_{\mathrm{abs}}^{j}(X,\mathbb{Z}{}_{p}(r))$ is the
group $H^{j}(X,\mathbb{Z}_{p}(r))$ defined in (\ref{n7}) (\cite{milneR2005}),
and%
\[
h^{i,j}(M(X))=h^{i,j}(X)\overset{\textup{{\tiny def}}}{=}\dim H^{j}%
(X,\Omega_{X}^{i}),
\]
because $R_{1}\otimes_{R}^{L}M(X)\simeq R\Gamma(X,\Omega_{X}^{\bullet})$ (see
\ref{b1}). Therefore, when $X$ is projective, Theorem \ref{b04} becomes the
$p$-part of Theorem 0.1 of \cite{milne1986}.

\subsection{Rigid cohomology}

Before considering more general algebraic varieties, we briefly review the
theory of rigid cohomology. This was introduced in the 1980s by Pierre
Berthelot as a common generalization of crystalline and Washnitzer-Monsky
cohomology. The book \cite{lestum2007} is a good reference for the
foundations. We write $H_{\mathrm{rig}}^{i}(X)$ (resp. $H_{\mathrm{rig},c}%
^{i}(X)$) for the rigid cohomology (resp. rigid cohomology with compact
support) of a variety $X$ over a perfect field $k$.

\begin{plain}
\label{r1}Both $H_{\mathrm{rig}}^{i}(X)$ and $H_{\mathrm{rig},c}^{i}(X)$ are
$F$-isocrystals over $k$; in particular, they are finite-dimensional vector
spaces over $K$. Cohomology with compact support is contravariant for proper
maps and covariant for open immersions; ordinary cohomology is contravariant
for all regular maps. The K\"{u}nneth theorem is true for both cohomology
theories. (\cite{berthelot1997a}, \cite{berthelot1997b}, \cite{gk2002}).
\end{plain}

\begin{plain}
\label{r2}When $X$ is smooth complete variety,%
\[
H_{\mathrm{rig}}^{i}(X)\simeq H_{\mathrm{crys}}^{i}(X)_{\mathbb{Q}{}}%
\]
(canonical isomorphism of $F$-isocrystals). (\cite{berthelot1986}).
\end{plain}

\begin{plain}
\label{r3}Let $U$ be an open subvariety of $X$ with closed complement $Z$;
then there is a long exact sequence%
\[
\cdots\rightarrow H_{\mathrm{rig},c}^{i}(U)\rightarrow H_{\mathrm{rig},c}%
^{i}(X)\rightarrow H_{\mathrm{rig},c}^{i}(Z)\rightarrow\cdots
\]
(\cite{berthelot1986}, 3.1).
\end{plain}

\begin{plain}
\label{r4}Rigid cohomology is a Bloch-Ogus theory. In particular, there is a
theory of rigid homology and cycle class maps. (\cite{petr2003}.)
\end{plain}

\begin{plain}
\label{r5}Rigid cohomology is a mixed-Weil cohomology theory and hence factors
through the triangulated category of complexes of mixed motives.
(\cite{cd2012, cd2013}).
\end{plain}

\begin{plain}
\label{r6}Rigid cohomology satisfies proper cohomological descent
(\cite{tsuz2003}).
\end{plain}

\begin{plain}
\label{r7}Rigid cohomomology (with compact support) can be described in terms
of the logarithmic de Rham-Witt cohomology of smooth simplicial schemes
(Lorenzon, Mokrane, Tsuzuki, Shiho, Nakkajima). We explain this below.
\end{plain}

\begin{plain}
\label{r8}When $k$ is finite, say, $k=\mathbb{F}{}_{p^{a}}$,%
\[
Z(X,t)=\det(1-F^{a}t\mid H_{\mathrm{rig},c}^{j}(X))^{(-1)^{j+1}}%
\]
(\cite{etesse1993}).
\end{plain}

\begin{plain}
\label{r9}When $k$ is finite, the $F$-isocrystals $H_{\mathrm{rig}}^{i}(X)$
and $H_{\mathrm{rig},c}^{i}(X)$ are mixed; in particular, the eigenvalues of
$\Phi=F^{a}$ are Weil numbers.
\end{plain}

The functors $X\rightsquigarrow H_{\mathrm{rig}}^{i}(X)$ and
$X\rightsquigarrow$ $H_{\mathrm{rig},c}^{i}(X)$ arise from functors to
$\mathsf{D}_{\text{\textrm{iso}}}^{b}(K_{\sigma}[F])$, which we denote
$h_{\mathrm{rig}}(X)$ and $h_{\mathrm{rig},c}(X)$ respectively.

\subsection{Varieties with log structure}

Endow $S$ with a fine log structure, and let $X$ be a complete log-smooth log
variety of Cartier type over $S$ (\cite{kato1989}). In this situation,
Lorenzon\nocite{lorenzon2002} (2002, Theorem 3.1) defines a complex
$M(X)\overset{\textup{{\tiny def}}}{=}R\Gamma(X,W\Omega_{X}^{\bullet})$ of
graded $R$-modules, and proves that it lies in $D_{c}^{b}(R)$. Therefore,
Theorem \ref{b04} applies to $X$.

\subsection{Smooth varieties}

Let $V=X\smallsetminus E$ be the complement of a divisor with normal crossings
$E$ in a smooth complete variety $X$ of dimension $n$, and let $m_{X}$ be the
canonical log structure on $X$ defined by $E$,%
\[
m_{X}=\{f\in\mathcal{O}{}_{X}\mid f\text{ is invertible outside }E\}
\]
(\cite{kato1989}, 1.5). Then $(X,m_{X})$ is log-smooth (ibid. \S 3).

Define $M(V)\in\mathsf{D}_{c}^{b}(R)$ to be the complex of graded $R$-modules
attached to $(X,m_{X})$ as above,%
\[
M(V)=R\Gamma((X,m_{X}),W\Omega_{X}^{\bullet})=R\Gamma(X,W\Omega_{X}^{\bullet
}(\log E)).
\]
We caution that this definition of $M(V)$ uses the presentation of $V$ as
$X\smallsetminus E$ --- it is not known at present that $M(V)$ depends only on
$V$. However%
\[
H^{i}(s(M(V))_{\mathbb{Q}{}{}}\simeq H_{\mathrm{rig}}^{i}(V)
\]
(\cite{nakk2012}, 1.0.18, p.13), and so $s(M(V))_{\mathbb{Q}{}}$ is
independent of the compactification $X$ of $V$.

We define $M_{c}(V)$ to be the Tate twist of the dual of $MV)$:%
\[
M_{c}(V)=D(M(V))(-n)\text{.}%
\]
(see \ref{b6c}). From Berthelot's duality of rigid cohomology
(\cite{berthelot1997a}; \cite{nakk2008}, 3.6.0.1), we have the following
isomorphism of $F$-isocrystals%
\[
H_{\mathrm{rig},c}^{j}(V)\simeq\Hom_{K}(H_{\mathrm{rig}}^{2n-j}(V),K(-n)).
\]
It follows that
\begin{equation}
H^{j}(s(M_{c}(V))_{\mathbb{Q}{}}\simeq H_{\mathrm{rig},c}^{j}(V). \label{eq66}%
\end{equation}
We define%
\[
H_{c}^{j}(V,\mathbb{Z}{}_{p}(r))=\Hom(W,M_{c}(V)(r)[j])\text{.}%
\]

Now take $k=\mathbb{F}{}_{p^{a}}$. It follows from (\ref{r8}) and (\ref{eq66})
that%
\[
Z(V,t)=Z(M_{c}(V),t)\text{.}%
\]
Moreover,
\[
R_{1}\otimes_{R}^{L}M(V)\simeq R\Gamma(X,\Omega_{X}^{\bullet}(\log E)).
\]
(\cite{lorenzon2002}, 2.17, or \cite{nakk2008}, p.184). Therefore, in this
case, Theorem \ref{b04} becomes the following statement.

\begin{theorem}
\label{b05}Assume that $q^{r}$ is not a multiple root of the minimum
polynomial of $F^{a}$ acting on $H_{\mathrm{rig}}^{j}(V)$ for any $j$.

\begin{enumerate}
\item The groups $H_{c}^{j}(V,\mathbb{Z}_{p}(r))$ are finitely generated
$\mathbb{Z}{}_{p}$-modules, and the alternating sum of their ranks is zero.

\item The zeta function $Z(V,t)$ of $X$ has a pole at $t=q^{-r}$ of order%
\[
\rho=\sum\nolimits_{j}(-1)^{j+1}\cdot j\cdot\rank_{\mathbb{Z}{}_{p}}\left(
H_{c}^{j}(V,\mathbb{Z}{}_{p}(r))\right)  \text{.}%
\]

\item The cohomology groups of the complex%
\[
E(V,r)\colon\quad\cdots\rightarrow H_{c}^{j-1}(V,\mathbb{Z}{}_{p}%
(r))\rightarrow H_{c}^{j}(V,\mathbb{Z}{}_{p}(r))\rightarrow H_{c}%
^{j+1}(V,\mathbb{Z}{}_{p}(r))\rightarrow\cdots
\]
are finite, and the alternating product of their orders $\chi(V,\mathbb{Z}%
{}_{p}(r))$ satisfies%
\[
\left\vert \lim_{t\rightarrow q^{-r}}Z(V,t)\cdot(1-q^{r}t)^{\rho}\right\vert
_{p}^{-1}=\chi(V,\mathbb{Z}{}_{p}(r))\cdot q^{\chi(V,r)}%
\]
where $\chi(V,r)=\sum_{i\leq r,j}(-1)^{i+j}(r-i)h^{i,j}(V)$.
\end{enumerate}
\end{theorem}

We caution the reader that it is not known that every smooth variety $U$ can
be expressed as the complement of a normal crossings divisor in a smooth
complete variety.

\subsection{General varieties}

\subsubsection{Philosophy}

With each variety $V$ over $k$, there should be associated objects $M(V),$
$M_{c}(V)$, $M^{BM}(V)$, $M^{h}(V)$ in $\mathsf{D}_{c}^{b}(R)$ arising as the
$p$-adic realizations of the various motives of $V$. See the discussion
\cite{voevodsky2000}, pp.181--182.

At present, it does not seem to be known whether there exists a $W$-linear
cohomology theory underlying Berthelot's rigid cohomology, i.e., a cohomology
theory that gives finitely generated $W$-modules $H_{c}^{j}(V)$ stable under
$F$ with $\mathbb{Q}\otimes_{\mathbb{Z}{}}{}H_{c}^{j}(V)=H_{\mathrm{rig,c}%
}^{j}(V)$ for each variety $V$.

Deligne's technique of cohomological descent in Hodge theory has been
transplanted to rigid and log-de Rham Witt theory by the brave efforts of N.
Tsuzuki, A. Shiho, and Y. Nakkajima. While their results do not provide the
invariants of $V$ above, they are still sufficient for applications to zeta
functions. Even though $M_{c}(V)$ is the only relevant object for zeta values,
we consider both $M(V)$ and $M_{c}(V)$.

\subsubsection{The ordinary cohomology object $M(V)$}

Let $V$ be a variety of dimension $n$ over $k$ equipped with an embedding
$V\hookrightarrow V^{\prime}$ of $V$ into a proper scheme $V^{\prime}$. Then
(see \cite{nakk2012}, especially 1.0.18, p.13), there is a simplicial proper
hypercovering $(U_{\bullet},X_{\bullet})$ of $(V,V^{\prime})$ with
$X_{\bullet}$ a proper smooth simplicial scheme over $k$ and $U_{\bullet}$ the
complement of a simplicial strict divisor with normal crossings $E_{\bullet}$
on $X_{\bullet}$; moreover,%
\[
H_{\mathrm{rig}}^{i}(V)\simeq H^{i}(X_{\bullet},W\Omega_{X_{\bullet}}%
^{\bullet}(\log E_{\bullet}))_{\mathbb{Q}{}}.
\]

For each $j\geq0$,
\[
R\Gamma(X_{j},W\Omega_{Xj}^{\bullet}(\log E_{j}))\in D_{c}^{b}(R)
\]
(\cite{lorenzon2002}). As $\mathsf{D}_{c}^{b}(R)$ is a triangulated
subcategory of $D(R)$, this implies that%
\[
R\Gamma(X_{\leq d},W\Omega_{X_{\leq d}}^{\bullet}(\log E_{\leq d}%
))\in\mathsf{D}_{c}^{b}(R)
\]
for each truncation $X_{\leq d}$ of the simplicial scheme $X_{\bullet}$. The
inclusion $X_{\leq d}\rightarrow X_{\bullet}$ induces an isomorphism%
\[
H^{i}(X_{\bullet},W\Omega_{X_{\bullet}}^{\bullet}(\log E_{\bullet
}))_{\mathbb{Q}{}}\simeq H^{i}(X_{\leq d},W\Omega_{X_{\leq d}}^{\bullet}(\log
E_{\leq d}))_{\mathbb{Q}{}}%
\]
for all $i$ provided $d>(n+1)(n+2)$ because both terms are isomorphic to
$H_{\mathrm{rig}}^{i}(V)$. For the left hand side, this follows from 1.0.8 or
12.9.1 of \cite{nakk2012}. For the right hand side, we apply Theorem 3.5.4,
p.243, of \cite{nakk2008}: $H_{\mathrm{rig}}^{i}(V)=0$ for $i>2n$ and the
spectral sequence 3.5.4.1 degenerates at $E_{1}$, implying that only finitely
many $X_{j}$'s contribute to the rigid cohomology of $V$. The bound on $d$
comes from the arguments following the isomorphism 3.5.0.4 on p.242 ibid. See
also pp.122-125 of \cite{nakk2012}.

We let
\[
M(V)=R\Gamma(X_{\leq d},W\Omega_{X_{\leq d}}^{\bullet}(\log E_{\leq d}))
\]
for any integer $d>(n+1)(n+2)$. While $M(V)$ may depend on $d$, $X_{\bullet}$,
and the embedding into $V^{\prime}$, the object $s(M(V))_{\mathbb{Q}{}}$ is
independent of these choices up to canonical isomorphism because
$H^{i}(s(M(V))_{\mathbb{Q}{}})\simeq H_{\mathrm{rig}}^{i}(V)$. Recall that
$H_{\mathrm{rig}}^{i}(V)=0$ for $i>2n$.

We need to truncate because it is not clear that the object $R\Gamma
(X_{\bullet},W\Omega_{X_{\bullet}}^{\bullet}(\log E_{\bullet}))$ lies in
$\mathsf{D}_{c}^{b}(R)$.

\subsubsection{The cohomology object with compact support $M_{c}(V)$}

Let $V\hookrightarrow V^{\prime}$ be as in the last subsubsection. Let
$\iota\colon Z\hookrightarrow V^{\prime}$ denote the inclusion of the reduced
closed complement $Z$ of $V$. One can find a proper hypercovering $Y_{\bullet
}\rightarrow Z$ and a morphism $f\colon Y_{\bullet}\rightarrow X_{\bullet}$
lifting $\iota$. Applying the results of the previous subsubsection to $Z$ and
fixing an integer $d>(n+1)(n+2)$, we get $M(Z)$ and a map $f^{\ast}\colon
M(V^{\prime})\rightarrow M(Z)$. We define $M_{c}(V)[1]$ to be the mapping cone
of $f^{\ast}$. This is an object of $\mathsf{D}_{c}^{b}(R).$

\begin{lemma}
\label{b06}For all $V\hookrightarrow V^{\prime}$ as above,%
\[
H^{i}(s(M_{c}(V))\simeq H_{\mathrm{rig},c}^{i}(V)\,.
\]

\end{lemma}

\begin{proof}
As the map $f$ lifts $\iota$, the map
\[
f^{\ast}\colon H^{i}(s(M(V^{\prime}))_{\mathbb{Q}{}})\rightarrow
H^{i}(s(M(Z))_{\mathbb{Q}{}})
\]
can be identified with the map $\iota^{\ast}\colon H_{\mathrm{rig}}%
^{i}(V^{\prime})\rightarrow H_{\mathrm{rig}}^{i}(Z)$. But as $V^{\prime}$ and
$Z$ are proper, rigid cohomology is the same as rigid cohomology with compact
support. The lemma now follows from the long exact sequence (\ref{r3}).
\end{proof}

Combining the lemma with the result of Etesse-Le Stum above, we obtain that
the zeta function $Z(V,t)$ of $V$ is equal to the zeta function of $M_{c}(V)$.
Therefore, from Theorem \ref{a0} we obtain Theorem \ref{b05} for $V$.

\subsection{Application of strong resolution of singularities}

Geisser (2006)\nocite{geisser2006} has shown how the assumption of a strong
form of resolution of singularities leads to a definition of groups $H_{c}%
^{i}(V,\mathbb{Z}{}(r))$ for an arbitrary variety $V$ over $k$, which, when
$k$ is finite, are closely connected to special values of zeta functions. His
definition involves the eh-topology, where the coverings are generated by
\'{e}tale coverings and abstract blow-ups (ibid. 2.1).

We now sketch how his argument provides an object $M_{c}(V)\in\mathsf{D}%
_{c}^{b}(R)$. For a complete $V$, we define $M(V)=R\Gamma(V_{\mathrm{eh}}%
,\rho^{\ast}W\Omega_{V}^{\bullet})$ where $\rho^{\ast}$ denotes pullback from
eh-sheaves on the category of smooth varieties over $k$ to eh-sheaves on all
varieties over $k$. We show that $M(V)\in\mathsf{D}_{c}^{b}(R)$ by using
induction on the dimension of $V$. Resolution of singularities gives us a
proper map $V^{\prime}\rightarrow V$ inducing an isomorphism from an open
subvariety $U^{\prime}$ of the smooth variety $V^{\prime}$ onto an open
subvariety $U$ of $V$. Moreover, $U^{\prime}$ is the complement in $V^{\prime
}$ of a divisor with normal crossings, and so we can define $M(U^{\prime}%
)\in\mathsf{D}_{c}^{b}(R)$ as above (using the eh-topology). Now
$M(V)\in\mathsf{D}_{c}^{b}(R)$ because $M(U)\overset{\textup{{\tiny def}}%
}{=}M(U^{\prime})\in\mathsf{D}_{c}^{b}(R)$ and $M(V\smallsetminus
U)\in\mathsf{D}_{c}^{b}(R)$ (by induction).

To define $M_{c}(V)$ for an arbitrary $V$, choose a compactification
$V^{\prime}$ of $V$, and let%
\[
M_{c}(V)=\mathrm{\mathrm{Cone}}(M(V^{\prime})\rightarrow M(Z))[-1],\quad
Z\overset{\textup{{\tiny def}}}{=}V^{\prime}\smallsetminus V\text{.}%
\]
Clearly, $M_{c}(V)\in\mathsf{D}_{c}^{b}(R)$. The eh-topology is crucial for
proving that this definition is independent of the compactification (ibid.
3.4). Given $M_{c}(V)$, we define%
\[
H_{c}^{i}(V,\mathbb{Z}{}_{p}(r))=\Hom_{\mathsf{D}_{c}^{b}(R)}(W,M_{c}%
(V)(r)[i]).
\]
This agrees with Geisser's group tensored with $\mathbb{Z}{}_{p}$, because the
two agree for smooth complete varieties and satisfy the same functorial properties.

\subsection{Deligne-Mumford Stacks}

Olsson (2007, first three chapters) extends the theory of crystalline
cohomology to certain algebraic stacks. He also shows (ibid. Chapter 4) that
the crystalline definition (\cite{illusie1983}, 1.1(iv)) of the de Rham-Witt
complex can be extended to stacks. Let $\mathcal{S}/W$ be a flat algebraic
stack $\,$equipped with a lift of the Frobenius endomorphism from
$\mathcal{S}{}_{0}$ compatible with the action of $\sigma$ in $W$. Let
$\mathcal{X}{}\rightarrow\mathcal{S}{}$ be a smooth morphism of algebraic
stacks with $\mathcal{X}{}$ a Deligne-Mumford stack. Then $W\Omega
_{\mathcal{\mathcal{X}{}}/S}^{\bullet}$ is a complex of sheaves of $R$-modules
on $\mathcal{X}{}$, and there is a canonical isomorphism%
\begin{equation}
H^{j}(s(R\Gamma(W\Omega_{\mathcal{\mathcal{X}{}}/S}^{\bullet})))_{\mathbb{Q}%
{}}\simeq H_{\mathrm{crys}}^{j}(\mathcal{X}{}/W)_{\mathbb{Q}{}} \label{eq34}%
\end{equation}
(\cite{olsson2007}, 4.4.17). Under certain hypotheses on $\mathcal{S}{}$ and
$\mathcal{X}$ (ibid. 4.5.1), Ekedahl's criterion (see \ref{a6}) can be used to
show that $R\Gamma(W\Omega_{\mathcal{\mathcal{X}{}}/S}^{\bullet})\in
\mathsf{D}_{c}^{b}(R)$ (ibid. 4.5.19) and that (\ref{eq34}) is an isomorphism
of $F$-isocrystals.

Now assume that $k=\mathbb{F}{}_{q}$, $q=p^{a}$. The zeta function
$Z(\mathcal{X}{},t)$ of a stack $\mathcal{X}$ over $k$ is defined by
(\ref{eq3}), but with%
\[
N_{m}=\sum_{x\in\lbrack\mathcal{X}(\mathbb{F}{}_{q^{m}})]}\frac{1}%
{\#\Aut_{x}\mathbb{F}{}_{q^{m}}}%
\]
(see \cite{sun2012}, p.49). Assume that $\mathcal{X}{}$ is a Deligne-Mumford
stack over $S$ satisfying Olsson's conditions, and let $M(\mathcal{X}%
{})=R\Gamma(W\Omega_{\mathcal{\mathcal{X}{}}/S}^{\bullet})\in\mathsf{D}%
_{c}^{b}(R)$. From (\ref{eq34}), we see that%
\[
Z(M(\mathcal{X}{}),t)=\prod\nolimits_{j}\det(1-F^{a}t\mid H_{\mathrm{crys}%
}^{j}(\mathcal{X}{}/W)_{\mathbb{Q}{}})^{(-1)^{j+1}}\text{.}%
\]
We expect that the two zeta functions agree (see ibid. 1.1 for the $\ell
$-version of this). Then Theorem \ref{b05} will hold for $\mathcal{X}{}$ with%
\[
H^{j}(\mathcal{X}{},\mathbb{Z}{}_{p}(r))\overset{\textup{{\tiny def}}%
}{=}\Hom_{\mathsf{D}_{c}^{b}(R)}(W,M(\mathcal{X}{})(r)[j])\text{.}%
\]

\subsection{Crystals}

Let $X$ be a smooth scheme over $S$, and let $E$ be a crystal on $X$. Etesse
(1988a, II, 1.2.5)\nocite{etesse1988a} defines a de Rham-Witt complex
$E\otimes W\Omega_{X/S}^{\bullet}$ on $X$, and, under some hypotheses on $X$
and $E$, he proves that $M(X,E)\in\mathsf{D}_{c}^{b}(R)$ (ibid., II, 1.2.7)
and that there is a canonical isomorphism
\[
H^{j}(R\Gamma(E\otimes W\Omega_{X/S}^{\bullet}))\simeq H_{\mathrm{crys}}%
^{j}(X/S,E)
\]
(ibid. II, 2.7.1). Let%
\[
M(X,E)=R\Gamma(E\otimes W\Omega_{X/S}^{\bullet}).
\]
When $k$ is finite, Theorem \ref{b04} for $M(X,E)$ becomes Theorem
(0.1)$^{\prime}$ of \cite{etesse1988b}.

\bibliographystyle{cbe}
\bibliography{D:/Current/mfrefs}

\bigskip James S. Milne,

Mathematics Department, University of Michigan, Ann Arbor, MI 48109, USA,

Email: jmilne@umich.edu

Webpage: \url{www.jmilne.org/math/}

\bigskip Niranjan Ramachandran,

Mathematics Department, University of Maryland, College Park, MD 20742, USA,

Email: atma@math.umd.edu,

Webpage: \url{www.math.umd.edu/~atma/}

\end{document}

%% file: defa.tex
\let\oldvec=\vec
\renewcommand{\vec}[1]{\reflectbox{\ensuremath{\oldvec{\reflectbox{\ensuremath{#1}}}}}}

\contentsmargin{2.0em}
\dottedcontents{section}[2.3em]{}{1.8em}{1pc}
\dottedcontents{subsection}[6.0em]{}{1.8em}{1pc}
\titleformat*{\subsection}{\large\scshape}
\titleformat*{\subsubsection}{\slshape}

\usepackage{titlesec}                                                                            
\titleformat*{\section}{\LARGE\bfseries}
\titleformat*{\subsection}{\Large\itshape}
\titleformat*{\subsubsection}{\scshape}
\titleformat*{\paragraph}{\itshape}

\usepackage{enumitem}
\setitemize[1]{label=$\diamond$\hspace{0.07in}}
\setlist{nolistsep}

\usepackage[nottoc]{tocbibind}

\usepackage{array}

\usepackage{tikz}
\usetikzlibrary{matrix,arrows,positioning,decorations.pathmorphing}
\usepackage{tikz-cd}

\usepackage{etex}

\usepackage{mathtools}

\usepackage[overload]{textcase}

\newcommand{\eb}[1]{{\itshape\bfseries#1}}
\renewcommand{\emph}{\eb}

\usepackage{natbib}
\let\cite\citealt

%% file: defs.tex
\newcommand{\bcomment}{}
\newcommand{\bfootnotesize}{\begin{footnotesize}}\newcommand\efootnotesize{\end{footnotesize}}
\newcommand{\bquote}{\begin{quote}}\newcommand\equote{\end{quote}}
\newcommand{\bsmall}{\begin{small}}\newcommand\esmall{\end{small}}
\newcommand{\btable}{\begin{table}}\newcommand{\etable}{\end{table}}

\newcommand{\edocument}{\end{document}}

\def\1{{1\mkern-7mu1}}

\DeclareMathOperator{\Aut}{Aut}

\DeclareMathOperator{\Coker}{Coker}

\DeclareMathOperator{\End}{End}

\DeclareMathOperator{\Ext}{Ext}

\DeclareMathOperator{\Gal}{Gal}

\DeclareMathOperator{\Hom}{Hom}
\DeclareMathOperator{\id}{id}

\DeclareMathOperator{\im}{Im}  

\DeclareMathOperator{\Ker}{Ker}

\DeclareMathOperator{\ord}{ord}

\DeclareMathOperator{\rank}{rank}

\DeclareMathOperator{\Spec}{Spec}

\newcommand{\Mod}{\mathsf{Mod}}

